\documentclass{amsart}
\usepackage[utf8]{inputenc} % direct use of umlauts
\usepackage{amsmath}
\usepackage{amstext}
\usepackage{amsfonts}
\usepackage{amsthm}
\usepackage{amssymb}
\usepackage{enumerate}
\usepackage{amsrefs}
\numberwithin{equation}{section}

\usepackage{color}

\newcommand{\adam}[1]{{\color{blue}{#1}}}
\newcommand{\stig}[1]{{\color{blue}{#1}}}

\usepackage{hyperref}
\hypersetup{colorlinks}

\newcommand{\R}{{\mathbf{R}}}
\newcommand{\E}{{\mathbf{E}}}
\newcommand{\N}{{\mathbf{N}}}
\newcommand{\D}{{\mathcal{D}}}
\newcommand{\dom}{{\mathrm{D}}}

\newcommand{\LB}{{\mathcal{L}}}

\newcommand{\diff}[1]{\,\mathrm{d}#1}

\newcommand{\C}{\mathcal{C}}
\newcommand{\Cb}{\mathcal{G}_{\mathrm{b}}}
\newcommand{\Cp}{\mathcal{G}_{\mathrm{p}}}

\theoremstyle{plain}

\newtheorem{definition}{Definition}[section]
\newtheorem{theorem}[definition]{Theorem}
\newtheorem{lemma}[definition]{Lemma}
\newtheorem{corollary}[definition]{Corollary}
\newtheorem{proposition}[definition]{Proposition}

\theoremstyle{definition}
\newtheorem{remark}[definition]{Remark}

\begin{document}
\title[Error analysis for semilinear stochastic Volterra equations]{Weak error analysis for semilinear stochastic Volterra equations with additive noise}

\author[A.~Andersson]{Adam Andersson}
\address{Adam Andersson\\
Technische \stig{U}niver\stig{s}it\"at Berlin\\
Institut f\"ur mathem\stig{a}tik\\
Secr.~MA 5-3, Strasse des 17.~Juni 136, DE--10623 Berlin, Germany}
\email{andersson@math.tu-berlin.de}

\author[M.~Kov\'{a}cs]{Mih\'{a}ly Kov\'{a}cs}
\address{Mih\'aly Kov\'acs\\
Department of mathematics and statistics\\
University of Otago\\
PO Box 56, Dunedin\\
9054\\
New Zealand}
\email{mkovacs@maths.otago.ac.nz}

\author[S.~Larsson]{Stig Larsson}
\address{Stig Larsson\\
Department of Mathematical Sciences\\
Chalmers University of Technology and University of Gothenburg\\
SE--412 96 Gothenburg\\
Sweden}
\email{stig@chalmers.se}


\begin{abstract}
  We prove a weak error estimate for the approximation in space and time of a semilinear stochastic Volterra integro-differential equation driven by additive space-time Gaussian noise. We treat this equation in an abstract framework, in which parabolic stochastic partial differential equations are also included as a special case. The approximation in space is performed by a standard finite element method and in time by an implicit Euler method combined with a convolution quadrature. The weak rate of convergence is proved to be twice the strong rate, as expected. Our convergence result concerns not only functionals of the solution at a fixed time but also more complicated functionals of the entire path and includes convergence of covariances and higher order statistics. The proof does not rely on a Kolmogorov equation. Instead it is based on a duality argument from Malliavin calculus.
\end{abstract}

\maketitle

\section{Introduction}
\label{sec1}
Let $(S_t)_{t\in[0,T]}$ be an evolution family of bounded, self-adjoint, linear operators on a separable Hilbert space $(H,\|\cdot\|,\langle\cdot,\cdot\rangle)$, not necessarily enjoying the semigroup property. Related to $(S_t)_{t\in[0,T]}$ is a densely defined, linear, self-adjoint, positive definite operator $A\colon \D(A)\subset H\to H$  with compact inverse. Let $(A^\alpha)_{\alpha\in\R}$ denote the fractional powers of $A$, which are well defined, let $(\dot{H}^\alpha)_{\alpha\in{R}}$ denote the spaces $\dot{H}^\alpha=\D(A^\alpha)$ for $\alpha\geq0$ with dual spaces $\dot{H}^{-\alpha}=(\dot{H}^\alpha)^*$. We assume that $(S_t)_{t\in[0,T]}$ is strongly differentiable with derivative $(\dot{S}_t)_{t\in[0,T]}$ and that there exist $\rho\in[1,2)$ and constants $(L_s)_{s\in[0,2]}$ so that
\begin{align}
\label{as:S}
  \big\|
    A^{\frac{\min(1,s)}\rho}
    S_t
    x
  \big\|
  +
  \big\|
    A^{\frac{s-1}\rho}
    \dot{S}_t
    x
  \big\|
  \leq
  L_s
  t^{-s}
  \|x\|
  ,\quad
  t\in(0,T]
  ,\
  x\in H
  ,\
  s\in[0,2].
\end{align}
If $(S_t)_{t\in[0,T]}$ is the analytic semigroup generated by $-A$, then \eqref{as:S} holds with $\rho=1$. If $(S_t)_{t\in[0,T]}$ is the solution operator $S_tx=Y_t^x$ of the Volterra equation
\begin{align*}
  \dot{Y}_t^x
  +
  \int_0^t
    b_{t-s}AY_s^x
  \diff{s}
  =0,
  \quad
  t\in(0,T]
  ;\quad
  Y_t^x=x,
\end{align*}
where $b\colon (0,\infty)\to\R$ is the Riesz kernel
$
  b_t
  =
  t^{\rho-2}
  /
  \Gamma(\rho-1)
$
for some $\rho\in(1,2)$, then $(S_t)_{t\in[0,T]}$ satisfies \eqref{as:S}.
The latter example is the main motivation of the present paper. In Subsection \ref{sec5:2} we verify \eqref{as:S} for slightly more general kernels $b$.

The main object of study in this paper is the stochastic evolution equation
\begin{align}
\label{eq:SVIE}
  X_t
  =
  S_tx_0+
  \int_0^tS_{t-s}F(X_s)\diff{s}
  +
  \int_0^tS_{t-s}\diff{W_s}
  ,\quad
  t\in[0,T].
\end{align}
The noise is generated by a cylindrical $Q$-Wiener process $W$ on a filtered probability space $(\Omega,\mathcal{F},(\mathcal{F}_t)_{t\in[0,T]},\mathbf{P})$ with positive semidefinite self-adjoint covariance operator $Q\in\LB(H)$, where the latter is the space of bounded linear operators on $H$. Let $H_0=Q^{\frac12}(H)$, and let $\LB_2$ and $\LB_2^0$ denote the spaces of Hilbert-Schmidt operators $H\to H$ and $H_0\to H$, respectively. The regularity of the noise is measured by a parameter $\beta\in(0,1/\rho]$, by assuming
\begin{align}
\label{as:Q}
  \big\|
    A^{\frac{\beta\rho-1}{2\rho}}
  \big\|_{\LB_2^0}
  =
  \big\|
    A^{\frac{\beta\rho-1}{2\rho}}Q^\frac12
  \big\|_{\LB_2}
  <
  \infty.
\end{align}
Under this assumption $X_t\in\dot{H}^\beta$, $\mathbf{P}$-almost surely. The smoothest case $\beta=1/\rho$ corresponds to trace class noise as \eqref{as:Q} reduces to $\|Q^{\frac12}\|_{\LB_2}=\sqrt{\mathrm{Tr}(Q)}<\infty$.

For Hilbert spaces $U$, $V$ the space $\Cb^k(U;V)$ consists of all, not necessarily bounded, functions $\phi\colon U\to V$, whose G\^ateaux derivatives of orders $1,\dots,k$ are bounded, symmetric and strongly continuous. The non-linear drift $F\colon H\to H$ is assumed to satisfy, for some $\delta\in[0,2/\rho)$,
\begin{align}
\label{as:F}
F\in\Cb^1(H;H)\cap\Cb^2(H;\dot{H}^{-\delta}).
\end{align}
This assumption includes interesting cases where $F\not\in \Cb^2(H;H)$, e.g., Nemytskii operators on $H=L^2(\dom)$ for a spatial domain $\dom\subset\R^d$, with $\delta>d/2$. The initial value $x_0$ is deterministic and satisfies
\begin{align}
\label{as:x0}
  x_0\in\dot{H}^3:=\D(A^\frac32).
\end{align}

In the present paper we study weak convergence of approximations of the solution of \eqref{eq:SVIE}. Our main example is the mild solution of the stochastic Volterra integro-differential equation
\begin{equation}
\begin{split}
\label{eq:SVDE}
  \diff{X_t}
  +
  \Big(
    \int_0^tb_{t-s}AX_s\diff{s}
  \Big)
  \diff{t}
  =
  F(X_t)\diff{t}
  +
  \diff{W_t}
  ,\ t\in[0,T]
  ;\quad
  X_0
  =x_0,
\end{split}
\end{equation}
where $b_t=t^{\rho-2}/\Gamma(\rho-1)$ as above or slightly more general. Discretization in time is performed by the backward Euler method and the convolution integral is approximated by a convolution quadrature. For spatial approximation either spectral or finite element approximation is considered. In the papers \cite{kovacs2013}, \cite{kovacs2014}, strong, respectively weak, convergence of numerical approximations were proven, for linear stochastic Volterra equations $(F=0)$. The deterministic error analysis needed for the present paper will be cited from these papers.

Another example to which our results apply is the mild solution of the parabolic stochastic evolution equation
\begin{equation}
\begin{split}
\label{eq:SPDE}
  \diff{X_t}
  +
  AX_t
  \diff{t}
& =
  F(X_t)\diff{t}
  +
  \diff{W_t}
  ,\ t\in[0,T];
  \quad X_0 =x_0.
\end{split}
\end{equation}
Approximation in time is performed by the backward Euler method and the same spatial approximation is considered as for \eqref{eq:SVDE}. Weak convergence analysis for \eqref{eq:SPDE} is well studied \cite{AnderssonKruseLarsson}, \cite{AnderssonLarsson}, \cite{Brehier}, \cite{Brehier3}, \cite{conus2014}, \cite{debussche2011}, \cite{hausenblas2003Weak}, \cite{hausenblas2010}, \cite{jentzen2015}, \cite{Wang}, \cite{Wang2014}, \cite{WangGan}. In contrast to \cite{AnderssonKruseLarsson} we allow the nonlinear drift $F$ to be a Nemytskii operator not only in one space dimension but also in two and three space dimensions, without imposing restrictions on the choice of the spatial approximation. We also consider a more general form of the weak error, see \eqref{eq:weak_distance} below. We thus present some new results also for \eqref{eq:SPDE}.

Let $K\in\N$ and $\varphi_i\colon H\rightarrow\R$, $i=1,\dots,K$, be twice G\^ateaux differentiable mappings of polynomial growth and $\nu_1,\dots,\nu_K$ finite Borel measures on $[0,T]$. We consider the generalized weak error
\begin{align}
\label{eq:weak_distance}
  \Big|
    \E\Big[
      \Phi
      \big(
        X
      \big)
      -
      \Phi
      \big(
        Y
      \big)
    \Big]
  \Big|,
  \quad
  \textrm{with}
  \quad
  \Phi(Z)
  =
  \prod_{i=1}^K
  \varphi_i
  \Big(
    \int_0^T
      Z_t
    \diff{\nu_{i,t}}
  \Big),
\end{align}
for $X,Y,Z\in \cap_{i=1}^K L_{\nu_i}^1(0,T;L^p(\Omega;H))$ with a suitable exponent $p\geq2$. In all the works we are aware of, \eqref{eq:weak_distance} is considered with $K=1$, $\nu_1=\nu=\delta_\tau$, where $\delta_\tau$ is the Dirac measure concentrated at $\tau$, for fixed $\tau\in(0,T]$. In that case $\E[\varphi(X_\tau)]$ is the solution to a Kolmogorov PDE, which is used in the analysis. Unfortunately, this is not true for $\E[\varphi(\int_0^T X_t\diff{\nu_t})]$. Moreover, Volterra equations are non-Markovian, so there is no Kolmogorov equation available for the analysis. Instead, we use another approach to analyze \eqref{eq:weak_distance} that was recently introduced in \cite{AnderssonKruseLarsson}. The approach relies on a duality argument with a Gelfand triple of refined Sobolev-Malliavin spaces. In \cite{AnderssonKruseLarsson} the technique was demonstrated in the Markovian setting of \eqref{eq:SPDE} and $\nu=\delta_\tau$. In the present paper we apply it in a setting where no other known approach applies.

Our main result, Theorem~\ref{thm:main}, shows convergence of the weak error of the form \eqref{eq:weak_distance} for abstractly defined approximations of the solution $X$ to \eqref{eq:SVIE}.  The general form of the functional $\Phi$ allows us  to prove convergence of approximations of covariances
\begin{align*}
  \mathrm{Cov}
  \big(
    \big\langle
      X_{t_1},\phi_1
    \big\rangle
    ,
    \big\langle
      X_{t_2},\phi_2
    \big\rangle
   \big),
   \quad
   \phi_1,\phi_2\in H,\ t_1,t_2\in(0,T],
\end{align*}
in Corollary~\ref{cor:main}. The generalization to higher order statistics is straightforward and omitted.

The paper is organized as follows: In Subsection \ref{sec2:1} we fix
the basic notation and in Subsection \ref{sec2:2} we recall the theory
of refined Sobolev-Malliavin spaces from
\cite{AnderssonKruseLarsson}. In Section \ref{sec3} we discuss
existence and uniqueness of solutions of \eqref{eq:SVIE} and prove
temporal H\"{o}lder regularity in the classical $L^p(\Omega;H)$-sense
and in the weaker sense of a dual Sobolev-Malliavin norm. In Section
\ref{sec:weak} we present an abstractly defined approximation scheme
for \eqref{eq:SVIE} and prove our main result on weak convergence,
Theorem~\ref{thm:main}. In addition, we prove strong convergence in
Theorem~\ref{thm:strong}, which is then used to establish Malliavin
regularity for the solution to \eqref{eq:SVIE} by a limiting
procedure. In Section \ref{sec5} we verify our abstract assumptions
for semilinear parabolic stochastic partial differential equations and
stochastic Volterra integro-differential equations.

\section{Preliminaries}
\label{sec2}
\subsection{Spaces of functions and operators}
\label{sec2:1}
Let $(U,\|\cdot\|_U,\langle\cdot,\cdot\rangle_U)$,
$(V,\|\cdot\|_V,\langle\cdot,\cdot\rangle_V)$ be separable Hilbert
spaces. Let $\LB(U;V)$ be the Banach space of all bounded linear
operators $U\rightarrow V$. We use the abbreviations $\LB(U)=\LB(U;U)$
and $\LB=\LB(H)$, where $H$ is the Hilbert space introduced in Section
\ref{sec1}. By $\LB_2(U;V)\subset\LB(U;V)$ we denote the subspace of
all Hilbert-Schmidt operators. It is a Hilbert space endowed with the
norm and inner product
\begin{align}
  \label{eq2:HSnorm}
  \|T\|_{\LB_2(U;V)}=\Big(\sum_{j\in\N}\|T u_j\|_V^2\Big)^\frac12,
  \quad \langle S,T\rangle_{\LB_2(U;V)}=\sum_{j\in\N}\langle
  Su_j,Tu_j\rangle_{V}.
\end{align}
Both are independent of the specific choice of ON-basis
$(u_j)_{j\in\N}\subset U$.

For $k\geq1$, let $\LB^{[k]}(U;V)$ be the Banach space of all
multilinear mappings $b\colon U^k\to V$, equipped with the norm
\begin{align*}
  \|
    b
  \|_{\LB^{[k]}(U;V)}
  =
  \sup_{u_1,\dots, u_k\in U}
  \frac{
    \|
      b\cdot(u_1,\dots,u_k)
    \|_V
  }{
    \|u_1\|_U
    \cdots
    \|u_k\|_U
  }.
\end{align*}
It is clear that $\LB^{[1]}(U;V)=\LB(U;V)$.

Denote by $\C(U;V)$ the space of all continuous mappings $U\to V$ and
further by $\C_{\mathrm{str}}(U;\LB^{[k]}(U;V))$ the space of strongly
continuous mappings $U\to \LB^{[k]}(U;V)$, i.e., mappings
$B\colon U\to \LB^{[k]}(U;V)$ such that for $u_1,\dots,u_k\in U$, the
mapping
\begin{align*}
  U\ni x
  \mapsto
  B(x)
  \cdot
  (u_1,\dots,u_k)
  \in V,
\end{align*}
is continuous. A function $\phi\colon U\to V$ is said to be $k$ times
G\^ateaux differentiable if the recursively defined derivatives,
$\phi^{(l)}\colon U^{l+1}\to V$, $l\in\{1,\dots,k\}$,
\begin{align*}
&\phi^{(l)}(x)\cdot (u_1,\dots,u_l)\\
&\quad
  =
  \lim_{\epsilon\to0}
  \frac{
    \phi^{(l-1)}(x+\epsilon u_l)
    \cdot(u_1,\dots,u_{l-1})
    -
    \phi^{(l-1)}(x)
    \cdot(u_1,\dots,u_{l-1})
  }\epsilon,
\end{align*}
exist for $u_1,\dots,u_l,x\in U $, $l\in\{1,\dots,k\}$, as limits in
$V$, where $\phi^{(0)}=\phi$. This class of functions is large and
fails to have natural properties, e.g., G\^ateaux differentiability
does not imply continuity and the multilinear mapping $\phi^{(l)}(x)$
may not be symmetric. We therefore introduce a smaller class, with
useful properties. For $k\geq1$, let
$\mathcal{G}^k(U;V)\subset\C(U;V)$ be the subset of all $k$-times
G\^ateaux differentiable mappings $\phi\in \C(U;V)$, whose
derivatives $\phi^{(l)}\in \C_{\mathrm{str}}(U;\LB^{[l]}(U;V))$,
$l\in\{1,\dots,k\}$, are symmetric. This is a weaker assumption than
requiring $\phi^{(l)}\in \C(U;\LB^{[l]}(U;V))$, $l\in\{1,\dots,k\}$,
which would be the same as assuming Fr\'echet differentiability. For
integers $k\in\{0,\dots,m\}$ and $\phi\in\mathcal{G}^k(U;V)$, let
\begin{align}
\label{eq:Cpkm}
  |\phi|_{\Cp^{k,m}(U;V)}
  =
  \sup_{u\in U}
  \frac{
    \|\phi^{(k)}(u)\|_{\LB^{[k]}(U;V)}
  }{
    (1+\|u\|_U^{m-k})
  },
\end{align}
and let $\Cp^{k,m}(U;V)$ be the space of $\phi\in\mathcal{G}^k(U;V)$
such that $|\phi|_{\Cp^{l,m}(U;V)}<\infty$ for $l\in\{1,\dots,k\}$. Let
$\Cp^\infty(U;V)$ be the space of all infinitely many times
differentiable mappings $\phi\colon U\to V$ such that $\phi$ and all
its derivatives satisfy a polynomial bound. Let $\Cb^k(U;V)$ denote
the space of $\phi\in\mathcal{G}^k(U;V)$ such that
\begin{align*}
  |\phi|_{\Cb^{l}(U;V)}
  =
  \sup_{u\in U}
  \|\phi^{(l)}(u)\|_{\LB^{[l]}(U;V)}
  <\infty
  ,\quad
l\in\{1,\dots,k\} .
\end{align*}
For $\phi\in\mathcal{G}^1(U;\R)$ we can identify the derivative with
the gradient $\phi'(u)\in U^*=U$, by the Riesz Representation
Theorem. For $m\geq 1$, $\phi\in\mathcal{G}_{\mathrm{p}}^{1,m}(U;V)$,
the map
$[0,1]\ni \lambda \mapsto \phi'(y+\lambda(x-y))\cdot (x-y)\in V $ is
continuous and Bochner integrable and therefore
\begin{equation}
\begin{split}
\label{eq:Taylor}
  \phi(x)
& =
  \phi(y)
  +
  \int_0^1\phi'(y+\lambda(x-y))\cdot(x-y)\diff{\lambda},
  \quad
  x,y\in U.
\end{split}
\end{equation}

By $\mathcal{M}_T$ we denote the space of all finite Borel measures on the interval $[0,T]$. For $\nu\in\mathcal{M}_T$ we write $|\nu|=\nu([0,T])$ and for a Banach space $V$ we let $L_\nu^p(0,T;V)$ be the Bochner space of $\nu$-measurable mappings $Z\colon[0,T]\rightarrow V$ such that
\begin{align*}
  \big\|
    Z
  \big\|_{L_\nu^p(0,T;V)}
  =
  \Big(
    \int_0^T
      \big\|
        Z_t
      \big\|_V^p
    \diff{\nu_t}
  \Big)^\frac1p
  <
  \infty,
\end{align*}
with the usual modification for $p=\infty$. When $\nu$ is Lebesgue measure we write $L^p(0,T;V)$.

The next lemma is used in the proof of Malliavin regularity by a limiting procedure in Proposition~\ref{lemma:Stability3}.
\begin{lemma}
\label{lemma:LimitReg}
Let $\mathcal{X}$, $\mathcal{Y}$ be separable Hilbert spaces such that the embedding
$
  \mathcal{X}
  \subset
  \mathcal{Y}
$
is continuous. If $x\in\mathcal{Y}$ and $(x_n)_{n\in\N}\subset\mathcal{X}$ satisfies $x_n\to x$ weakly in $\mathcal{Y}$ as $n\to\infty$ and $\sup_{n\in\N}\|x_n\|_{\mathcal{X}}<\infty$, then $x\in\mathcal{X}$.
\end{lemma}

\begin{proof}
Any closed ball in $\mathcal{X}$ is weakly compact and since $(x_n)_{n\in\N}$ is a bounded sequence in $\mathcal{X}$, there exists a subsequence $(x_{n_k})_{k\in\N}$ and $\tilde{x}\in\mathcal{X}$ such that $x_{n_k}\to\tilde{x}$ weakly in $\mathcal{X}$. Therefore $x_{n_k}\to \tilde{x}$ also in the weak topology of $\mathcal{Y}$ because $\mathcal{Y}^*\subset\mathcal{X}^*$ is continuous. By assumption $x_n\to x$ weakly in $\mathcal{Y}$, as $n\to\infty$, so $x=\tilde{x}\in\mathcal{X}$.
\end{proof}

We cite the following version of Gronwall's Lemma \cite{elliott1992}*{Lemma 7.1}.

\begin{lemma}
\label{lemma2:Gronwall}
Let $T>0$, $N\in\N$, $k=T/N$, and $t_n=nk$ for $0\leq n\leq N$. If $\varphi_1,\dots,\varphi_N\geq0$ satisfy for some $M_0,M_1 \ge 0$ and $\mu,\nu>0$ the inequality
\begin{align*}
  \varphi_n \leq
  M_0 \,(1 + t_n^{-1+\mu}) + M_1
  \,k\,\sum_{j=1}^{n-1}t_{n-j}^{-1+\nu}\varphi_j, \quad 1\leq n\leq N,
\end{align*}
then there exists a constant
$M_2=M_2(\mu,\nu,M_1,T)$ such that
\begin{align*}
  \varphi_n\leq M_0M_2\,(1 + t_n^{-1+\mu}), \quad 1\leq n\leq N.
\end{align*}
\end{lemma}

\subsection{The Wiener integral and Malliavin calculus}
\label{sec2:2}
Let
$
  (
  \Omega,
  \mathcal{F},
  (\mathcal{F}_t)_{t\in[0,T]},
  \mathbf{P}
  ),
$
be a filtered probability space, with Bochner spaces $L^p(\Omega;V)=L^p((\Omega,\mathcal{F},\mathbf{P});V)$, $p\in[1,\infty]$, $V$ being a Banach space. In the case $V=\R$ we write $L^p(\Omega)=L^p(\Omega;\R)$. Recall that $Q\in\LB(H)$ is a linear, self-adjoint and positive semidefinite operator. Let $H_0=Q^\frac12(H)$ be the Hilbert space endowed
with inner product
$
  \langle
    u,v
  \rangle_{H_0}
  =
  \langle
    Q^{-\frac12}u
    ,
    Q^{-\frac12}v
  \rangle
$,
where $Q^{-\frac12}$ denotes the pseudoinverse of $Q^\frac12$ if it is not injective. By $\LB_2^0=\LB_2(H_0;H)$ we denote the space of Hilbert-Schmidt operators $H_0\rightarrow H$. Let $W$ be a cylindrical $Q$-Wiener process on
$
  (
  \Omega,
  \mathcal{F},
  (\mathcal{F}_t)_{t\in[0,T]},
  \mathbf{P}
  )
$,
i.e.,
$W\in \LB(H_0;\C(0,T;L^2(\Omega)))$ and $((Wu)_t)_{t\in[0,T]}$ is an $(\mathcal{F}_t)_{t\in[0,T]}$-adapted real-valued Brownian motion for every $u\in H_0$ with
\begin{align*}
  \E
  \big[
    (Wu)_s
    \,
    (Wv)_t
  \big]
  =
  \min(s,t)
  \langle
    u,v
  \rangle_{H_0},
  \quad
  u,v\in H_0,
  \
  s,t\in[0,T].
\end{align*}

The stochastic Wiener integral
\begin{align*}
\int_0^T\Phi_t\diff{W_t},\quad \Phi\in L^2(0,T;\LB_2^0),
\end{align*}
is a random variable in $L^p(\Omega;H)$, $p\in[2,\infty)$. It can be
defined in various ways and its basic properties are not hard to
derive, we refer to \cite{daprato1992, roeckner2007, UMD}. We cite the
following consequence of the Burkholder inequality
\cite{daprato1992}*{Lemma 7.2}, for deterministic integrands and
$p\geq2$,
\begin{align}
\label{ineq:Burkholder}
  \Big\|
    \int_0^T\Phi_t\diff{W_t}
  \Big\|_{L^p(\Omega;H)}
  \leq
  \frac{p(p-1)}2
  \big\|
    \Phi
  \big\|_{L^2(0,T;\LB_2^0)},
  \quad
  \Phi\in L^2(0,T;\LB_2^0).
\end{align}
By taking $H=\R$ and noting the isomorphisms $H_0\cong H_0^*\cong\LB_2(H_0;\R)$ we see that a function $\phi\in L^2(0,T;H_0)$ defines an integrand in $L^2(0,T;\LB_2(H_0;\R))$ for the stochastic integral and the integral $\int_0^T\phi_t\diff{W_t}\in L^2(\Omega)$ is real-valued. As $L^p(0,T;H_0)\subset L^2(0,T;H_0)$ for $p\geq2$ the stochastic integral is well defined for $\phi\in L^p(0,T;H_0)$.

We now recall some concepts from Malliavin calculus introduced in \cite{AnderssonKruseLarsson}. For $q\in[2,\infty]$ let $\mathcal{S}^q(\R)$ be the class of smooth cylindrical random variables of the form
\begin{align*}
& F
  =
  f
  \Big(
    \int_0^T
        \phi_{1,s}
    \diff{W_s}
    ,
    \dots
    ,
    \int_0^T
        \phi_{n,s}
    \diff{W_s}
  \Big),\\
& f\in\Cp^\infty(\R^n;\R),
  \ (\phi_k)_{k=1}^n\subset L^q(0,T;H_0),
  \ n\in\N.
\end{align*}
For $F\in\mathcal{S}^q(\R)$ with the above representation, we define the Malliavin derivative
\begin{align*}
& \big(
    D_tF
  \big)_{t\in[0,T]}
  =
  \Bigg(
    \sum_{j=1}^n
    \partial_j f
    \Big(
      \int_0^T
          \phi_{1,s}
      \diff{W_s}
      ,
      \dots
      ,
      \int_0^T
          \phi_{n,s}
      \diff{W_s}
    \Big)\otimes\phi_{j,t},
  \Bigg)_{t\in[0,T]}.
\end{align*}
Let $V$ be a separable Hilbert space. We define $\mathcal{S}^q(V)$ to be the space of all $V$-valued random variables of the form $Y=\sum_{i=1}^m v_i\otimes F_i$ with $(v_i)_{i=1}^m\subset V$, $(F_i)_{i=1}^m\subset \mathcal{S}^q(\R)$, $m\in\N$. The Malliavin derivative of $Y\in\mathcal{S}^q(V)$ of the above form is given by $D_tY=\sum_{i=1}^m v_i\otimes D_tF_i$. As $(D_tF_i)_{t\in[0,T]}$ is an $H_0$-valued process, $(D_tY)_{t\in[0,T]}$ is a $V\otimes H_0=\LB_2(H_0;V)$-valued process.

For $p\in[2,\infty)$, $q\in[2,\infty]$,
$\mathcal{S}^q(V)\subset L^p(\Omega;V)$ is dense by \cite{AnderssonKruseLarsson}*{Lemma
3.1} and the operator
$D\colon \mathcal{S}^q(V)\rightarrow
L^p(\Omega;L^q(0,T;\LB_2(H_0;V)))$
is closable by \cite{AnderssonKruseLarsson}*{Lemma 3.2}. Let
$\mathbf{M}^{1,p,q}(V)$ denote the closure of $\mathcal{S}^q(V)$ with
respect to the norm
\begin{align*}
  \|Y\|_{\mathbf{M}^{1,p,q}(V)}=\Big(\|Y\|_{L^p(\Omega;V)}^p
  +\|DY\|_{L^p(\Omega;L^q(0,T;\LB_2(H_0;V)))}^p\Big)^\frac1p.
\end{align*}
We also use the corresponding seminorm
$|Y|_{\mathbf{M}^{1,p,q}(V)}=\|DY\|_{L^p(\Omega;L^q(0,T;\LB_2(H_0;V)))}$.
The spaces $\mathbf{M}^{1,p,q}(V)$ are Banach spaces, densely embedded
into $L^2(\Omega;V)$. Thus,
$\mathbf{M}^{1,p,q}(V)\subset L^2(\Omega;V)\subset
\mathbf{M}^{1,p,q}(V)^*$
is a Gelfand triple. By \cite{AnderssonKruseLarsson}*{Theorem 3.5} the
following inequality holds for $p\in[2,\infty)$, $q\in[2,\infty]$ with
$\tfrac1q+\tfrac1{q'}=1$:
\begin{align}
\label{ineq:BurkholderDual}
  \Big\|
    \int_0^T\Phi_t\diff{W_t}
  \Big\|_{\mathbf{M}^{1,p,q}(V)^*}
  \leq
  \big\|
    \Phi
  \big\|_{L^{q'}(0,T;\LB_2(H_0;V))}
  ,
  \quad
  \Phi\in L^2(0,T;\LB_2(H_0;V)).
\end{align}
What makes this duality theory useful is the possibility of taking
$q'$ close to $1$, c.f., \eqref{ineq:Burkholder} where the exponent is
$2$. We only need \eqref{ineq:Burkholder} and
\eqref{ineq:BurkholderDual} for deterministic integrands but remark
that \cite{AnderssonKruseLarsson}*{Theorem 3.5} allows $\Phi$ to be
random and only Skorohod integrability is required. Following
\cite{AnderssonKruseLarsson} we refer to $\mathbf{M}^{1,p,q}(H)$ for
$q>2$ as refined Sobolev-Malliavin spaces. The spaces
$\mathbf{M}^{1,p,2}(V)$ are classical Sobolev-Malliavin spaces, often
denoted $\mathbf{D}^{1,p}(V)$. For $p=q$ we write $\mathbf{M}^{1,p}(V):=\mathbf{M}^{1,p,p}(V)$.

    We next state a modified version of
    \cite{AnderssonKruseLarsson}*{Lemma~3.10}. It provides a local
    Lipschitz bound that enables us to prove an error estimate in the
    $\mathbf{M}^{1,p}(H)^*$-norm by a Gronwall argument in Lemma
    \ref{thm:strong2} below. More precisely,
    \cite{AnderssonKruseLarsson}*{Lemma~3.10} is a local Lipschitz
    bound from $\mathbf{G}^{1,p}(U)^*$ to $\mathbf{G}^{1,p}(V)^*$ for
    mappings $\sigma\in\Cb^2(U;V)$, where
    $\mathbf{G}^{1,p}(U)=\mathbf{M}^{1,p}(U)\cap L^{2p}(\Omega;U)$.
    The Lipschitz constant depends on the
    $\mathbf{M}^{1,2p,p}(U)$-norms of the random variables.  By
    restriction to random variables in $\mathbf{M}^{1,p}(U)$ with
    Malliavin derivative bounded over $\Omega$,
    Lemma~\ref{lemma:Lipschitz} provides a more natural bound,
    obviating the need for the spaces $\mathbf{G}^{1,p}(V)$. The
    Lipschitz constant now depends on the
    $\mathbf{M}^{1,\infty,p}(U)$-seminorm. It is proved in the same
    way as \cite{AnderssonKruseLarsson}*{Lemma~3.10}, by application
    of a modified version of \cite{AnderssonKruseLarsson}*{Lemma~3.8},
    based on H\"older's inequality with exponents $1$, $\infty$
    instead of $2$, $2$. We omit the details.  In the following
    Lemma~\ref{lemma:operator_norm} we cite parts of
    \cite{AnderssonKruseLarsson}*{Lemma~3.9}.

\begin{lemma}
\label{lemma:Lipschitz}
Let $U,V$ be separable Hilbert spaces, $\sigma\in\Cb^2(U;V)$, and
$p\in[2,\infty)$. For $Y^1,Y^2\in\mathbf{M}^{1,p}(U)$ with  
$DY^1,DY^2\in L^{\infty}(\Omega;L^p(0,T;\LB(H_0;U)))$, it holds that 
\begin{align*}
  \big\|
    \sigma(Y^1)-\sigma(Y^2)
  \big\|_{\mathbf{M}^{1,p}(V)^*}
& \leq
  \max
  \big(
    |\sigma|_{\Cb^1(U;V)},
    |\sigma|_{\Cb^2(U;V)}
  \big)\\
& \quad
  \times
  \Big(1+
    \sum_{i=1}^2
    \big|
      Y^i
    \big|_{\mathbf{M}^{1,\infty,p}(U)}
  \Big)
  \big\|
    Y^1-Y^2
  \big\|_{\mathbf{M}^{1,p}(U)^*}.
\end{align*}
\end{lemma}

%We finally cite parts of \cite{AnderssonKruseLarsson}*{Lemma 3.9}. It is used in the proof of
%Lemma~\ref{thm:strong2}

\begin{lemma}\label{lemma:operator_norm}
Let $p\in[2,\infty)$, $q\in[2,\infty]$.
Then for all $S\in \LB(H)$, $Y\in L^2(\Omega;H)$ it holds that
$
  \|
    S Y
  \|_{\mathbf{M}^{1,p,q}(H)^*}
  \leq
  \|S\|_{\LB(H)}
  \|Y\|_{\mathbf{M}^{1,p,q}(H)^*}.
$
\end{lemma}

\section{Existence, uniqueness and regularity}
\label{sec3}
Throughout this section we assume that \eqref{as:S}, \eqref{as:Q}--\eqref{as:x0} hold with $\rho\in[1,2)$, $\beta\in(0,1/\rho]$.
We begin by proving existence, uniqueness, and Malliavin regularity of the solution of \eqref{eq:SVIE}. Recall that two stochastic processes $X^1,X^2$ are modifications of each other if for all $t\in[0,T]$ it holds that $\mathbf{P}(X_t^1\neq X_t^2)=0$.
\begin{proposition}
\label{prop:existence}
There exists an, up to modification, unique stochastic process $X\colon[0,T]\times\Omega\to H$ such that $X\in\C(0,T;L^p(\Omega;H))$ for $p\in[2,\infty)$ and such that $X\in \C(0,T;\mathbf{M}^{1,p,q}(H))$ for $p\in[2,\infty)$, $q\in[2,\tfrac2{1-\rho\beta})$ and which satisfies equation \eqref{eq:SVIE} $\mathbf{P}$-a.s..
\end{proposition}

\begin{proof}
  Existence is proved by a standard application of Banach's Fixed
  Point Theorem, see, e.g., \cite{jentzen2010b}*{Theorem 1} or
  \cite{BaeumerGeissertKovacs2014}*{Theorem 3.3}. We note that for
  proving existence and uniqueness in $\C(0,T;L^p(\Omega;H))$ it is
  not crucial whether $(S_t)_{t\in[0,T]}$ is a semigroup or not. For
  the $\C(0,T;\mathbf{M}^{1,p,q}(H))$ regularity, see Proposition
  \ref{lemma:Stability3} below.
\end{proof}

The next proposition states the temporal H\"{o}lder regularity of $X$ in the $L^p(\Omega;H)$ and $\mathbf{M}^{1,p,q}(H)^*$ norms. Note that the H\"{o}lder exponent in the $\mathbf{M}^{1,p,q}(H)^*$ norm is twice that in the $L^p(\Omega;H)$ norm.

\begin{proposition}
\label{prop:Holder}
Let $X$ be the solution to \eqref{eq:SVIE}. For $\gamma\in(0,\beta)$, $p\geq2$, $q=\tfrac2{1-\rho\gamma}$, there exists $C>0$ such that
\begin{align*}
  \big\|
    X_{t_2}-X_{t_1}
  \big\|_{L^p(\Omega;H)}
& \leq
  C
  \big|
    t_2-t_1
  \big|^{\frac{\rho\gamma}2},
  \quad
  t_1,t_2\in[0,T],\\
  \big\|
    X_{t_2}-X_{t_1}
  \big\|_{\mathbf{M}^{1,p,q}(H)^*}
& \leq
  C
  \big|
    t_2-t_1
  \big|^{\rho\gamma},
  \quad
  t_1,t_2\in[0,T].
\end{align*}

\end{proposition}

\begin{proof}
Fix $\gamma\in(0,\beta)$, $p\geq2$. In order to treat both cases simultaneously we define $V_2=L^p(\Omega;H)$, $c_{p,2}=p(p-1)/2$, and $V_r=\mathbf{M}^{1,p,r}(H)^*$, $c_{p,r}=1$ for $r\in(2,\infty]$. In view of \eqref{ineq:Burkholder} and \eqref{ineq:BurkholderDual} it holds that
\begin{align}
\label{ineq:Vq}
  \Big\|
    \int_0^T
      \Phi_t
    \diff{W_t}
  \Big\|_{V_r}
  \leq
  c_{p,r}
  \big\|
    \Phi
  \big\|_{L^{r'}(0,T;\LB_2^0)},
  \quad
  \Phi\in L^2(0,T;\LB_2^0),\
  r\in[2,\infty],
\end{align}
where $\tfrac1r+\tfrac1{r'}=1$. Let $t_2>t_1$. The difference $X_{t_2}-X_{t_1}$ can be written in the form
\begin{align*}
  X_{t_2}-X_{t_1}
& =
  \big(
    S_{t_2}-S_{t_1}
  \big)x_0
  +
  \int_0^{t_1}
    \big(
      S_{t_2-s}-S_{t_1-s}
    \big)
    F(X_s)
  \diff{s}
  +
  \int_{t_1}^{t_2}
    S_{t_2-s}F(X_s)
  \diff{s}\\
& \quad
  +
  \int_0^{t_1}
    \big(
      S_{t_2-s}-S_{t_1-s}
    \big)
  \diff{W_s}
  +
  \int_{t_1}^{t_2}
    S_{t_2-s}
  \diff{W_s}.
\end{align*}
Taking $V_r$-norms, using the continuous embeddings
$
  H\subset
  L^p(\Omega;H)\subset
  L^2(\Omega;H)\subset
  \mathbf{M}^{1,p,r}(H)^*
$, yields
\begin{align*}
& \big\|
    X_{t_2}-X_{t_1}
  \big\|_{V_r}
  \leq
  \big\|
    \big(
      S_{t_2}-S_{t_1}
    \big)x_0
  \big\|\\
& \qquad
  +
  \Big\|
    \int_0^{t_1}
      \big(
        S_{t_2-s}-S_{t_1-s}
      \big)
      F(X_s)
    \diff{s}
  \Big\|_{L^p(\Omega;H)}
  +
  \Big\|
    \int_{t_1}^{t_2}
      S_{t_2-s}F(X_s)
    \diff{s}
  \Big\|_{L^p(\Omega;H)}\\
& \qquad
  +
  \Big\|
    \int_0^{t_1}
      \big(
        S_{t_2-s}-S_{t_1-s}
      \big)
    \diff{W_s}
  \Big\|_{V_r}
  +
  \Big\|
    \int_{t_1}^{t_2}
      S_{t_2-s}
    \diff{W_s}
  \Big\|_{V_r}.
\end{align*}
First, by \eqref{as:S} and \eqref{as:x0}, we obtain
\begin{align*}
  \big\|
    \big(
      S_{t_2}-S_{t_1}
    \big)x_0
  \big\|
& =
  \Big\|
    \int_{t_1}^{t_2}
      \dot{S}_t
      A^{-\frac1\rho}
      A^\frac1\rho
      x_0
    \diff{t}
  \Big\|
  \leq
  L_0
  \big\|
    A^{\frac1\rho}x_0
  \big\|
  (t_2-t_1).
\end{align*}
It is straightforward to show that the terms containing $F$ are
bounded up to a constant by $|t_2-t_1|^{1-\epsilon}$, and $|t_2-t_1|$
respectively, for every $\epsilon\in(0,1)$. For the case $\rho=1$ see
the proof of \cite{AnderssonKruseLarsson}*{Proposition 3.11}.

By \eqref{ineq:Vq}, \eqref{as:Q}, and \eqref{as:S} we get
\begin{align*}
& \Big\|
    \int_0^{t_1}
      \big(
        S_{t_2-s}-S_{t_1-s}
      \big)
    \diff{W_s}
  \Big\|_{V_r}\\
& \quad
  \leq
  c_{p,r}
  \Big(
    \int_0^{t_1}
      \big\|
        \big(
          S_{t_2-s}-S_{t_1-s}
        \big)
        A^{\frac{1-\beta\rho}{2\rho}}
      \big\|_{\LB}^{r'}
      \big\|
        A^{\frac{\beta\rho-1}{2\rho}}
      \big\|_{\LB_2^0}^{r'}
    \diff{s}
  \Big)^\frac1{r'}\\
& \quad
  \leq
  c_{p,r}
  \big\|
    A^{\frac{\beta\rho-1}{2\rho}}
  \big\|_{\LB_2^0}
  \Big(
    \int_0^{t_1}
      \Big(
        \int_{t_1}^{t_2}
          \|\dot{S}_{t-s}A^{\frac{(3-\beta\rho)/2-1}\rho}\|_{\LB}
        \diff{t}
      \Big)^{r'}
    \diff{s}
  \Big)^\frac1{r'}\\
& \quad
  \leq
  c_{p,r}
  \big\|
    A^{\frac{\beta\rho-1}{2\rho}}
  \big\|_{\LB_2^0}
  L_{\frac{3-\beta\rho}2}
  \Big(
    \int_0^{t_1}
      \Big(
        \int_{t_1}^{t_2}
          (t-s)^{-\frac{3-\beta\rho}2}
        \diff{t}
      \Big)^{r'}
    \diff{s}
  \Big)^\frac1{r'}.
\end{align*}
Bounding the integrals yields, for $\eta\in(0,1/\rho)$ to be chosen,
\begin{align*}
& \Big(
    \int_0^{t_1}
      \Big(
        \int_{t_1}^{t_2}
          (t-s)^{-\frac{3-\beta\rho}2}
        \diff{t}
      \Big)^{r'}
    \diff{s}
  \Big)^\frac1{r'}\\
& \quad
  \leq
  \Big(
    \int_0^{t_1}
      \Big(
        (t_1-s)^{-\frac{1-(\beta-2\eta)\rho}2}
        \int_{t_1}^{t_2}
          (t-t_1)^{-1+\eta\rho}
        \diff{t}
      \Big)^{r'}
    \diff{s}
  \Big)^\frac1{r'}\\
& \quad
  =
  \frac{
    (t_2-t_1)^{\eta\rho}
  }
  {
    \eta\rho
  }
  \Big(
    \int_0^{t_1}
      (t_1-s)^{-\frac{r}{r-1}\frac{1-(\beta-2\eta)\rho}2}
    \diff{s}
  \Big)^\frac{r-1}{r}.
\end{align*}
For $r=q=2/(1-\gamma\rho)$ and $\eta<(\beta+\gamma)/2$, the exponent is
\begin{align*}
  \frac{r}{r-1}
  \frac{1-(\beta-2\eta)\rho}2
  =
  \frac{1-\beta\rho+2\eta\rho}{1+\rho\gamma}<1.
\end{align*}
In particular, we can take $\eta=\gamma$ as required since $\gamma<\beta$. For $r=2$, the analogous condition is $\eta<\beta/2$ and we can take $\eta=\gamma/2$. Next, similarly,
\begin{align*}
  \Big\|
    \int_{t_1}^{t_2}
      S_{t_2-s}
    \diff{W_s}
  \Big\|_{V_r}
& \leq
  c_{p,r}
  \Big(
    \int_{t_1}^{t_2}
      \big\|
        S_{t_2-s}
        A^{\frac{1-\beta\rho}{2\rho}}
      \big\|_{\LB}^{r'}
      \big\|
        A^{\frac{\beta\rho-1}{2\rho}}
      \big\|_{\LB_2^0}^{r'}
    \diff{s}
  \Big)^\frac1{r'}\\
& \leq
  c_{p,r}
  L_{\frac{1-\beta\rho}{2}}
  \big\|
    A^{\frac{\beta\rho-1}{2\rho}}
  \big\|_{\LB_2^0}^{r'}
  \Big(
    \int_{t_1}^{t_2}
      (t_2-s)^{-\frac{r}{r-1}\frac{1-\beta\rho}2}
    \diff{s}
  \Big)^\frac{r-1}{r}\\
& \leq
  (t_2-t_1)^{\frac{r-1}r-\frac{1-\beta\rho}2}.
\end{align*}
For $r=q=2/(1-\gamma\rho)$ we have the H\"{o}lder exponent
\begin{align*}
  \frac{r-1}r
  -
  \frac{1-\beta\rho}2
  =
  \frac{\rho(\beta+\gamma)}2
  >
  \gamma\rho,
\end{align*}
and for $r=2$ the H\"{o}lder exponents equals $\beta\rho/2>\gamma\rho/2$.
\end{proof}

\section{Weak and strong convergence}
\label{sec:weak}

This section contains our main result and its proof. Theorem
\ref{thm:main} states a weak error estimate for abstractly defined
approximations of quantities of the form
$\E[\Phi(X)]=\E[\prod_{i=1}^K \varphi_i(\int_0^T X_t \diff{\nu_t^i})]$
for $(\nu^i)_{i=1}^K\subset \mathcal{M}_T$,
$(\varphi_i)_{i=1}^K\subset \mathcal{G}_{\mathrm{p}}^{2,m}(H;\R)$,
$m\geq2$, and $X$ being the solution to \eqref{eq:SVIE}. Theorem
\ref{thm:strong} provides a strong error estimate for approximations
of $X$. For parabolic problems, weak convergence, more precisely,
convergence of approximations of $\E[\varphi(X_t)]$ for fixed
$t\in[0,T]$ has been considered \cite{AnderssonKruseLarsson}, and for
Volterra equations in \cite{kovacs2014} but only in the linear case
$F=0$. To the best of our knowledge the more general convergence in
Theorem \ref{thm:main} is new in both cases. The rate of convergence
for $\E[\Phi(X)]$ is twice the strong rate as expected. We begin by
presenting a family of abstractly defined approximations.

\subsection{Approximation}
\label{sec4:1}
Assume that \eqref{as:S}, \eqref{as:Q}--\eqref{as:x0} hold. Let
$(V_h)_{h\in(0,1)}$ be a family of finite-dimensional subspaces of $H$
and let $P_h\colon H\rightarrow V_h$ be the orthogonal projector. Let
$k\in(0,1)$ and $t_n=nk$, $n=0,\dots,N$, where $t_{N}< T\leq
t_{N}+k$.
Let $(B^{h,k})_{h,k\in(0,1)}$ be a family of operator-valued functions
$B^{h,k}\colon \{0,\dots,N\}\rightarrow \LB(H;V_h)$ such that
$B_n^{h,k}=B_n^{h,k}P_h$, and let $(A_h)_{h\in(0,1)}$ be a collection
of linear operators $A_h\colon V_h\rightarrow V_h$ such that for
$n=1,\dots,N$ it holds that
\begin{align}
\label{as:B}
  \big\|
    A_h^{\frac{s}\rho}
    B_n^{h,k}
    x
  \big\|
& \leq
  L_s
  t_n^{-s}
  \|x\|,
  \quad
  x\in H
  ,\
  0\leq s\leq1,
\end{align}
with the same constants $(L_s)_{s\in[0,1]}$ as in \eqref{as:S}. For other constants $(K_\epsilon)_{\epsilon\in(0,\infty)}$ and $(R_s)_{s\in[0,1]}$, let the corresponding error operator $(E^{h,k})_{h,k\in(0,1)}$, given by $E_n^{h,k}=S_{t_n}-B_n^{h,k}$ for $n=0,\dots,N$, satisfy the smooth data error estimate
\begin{align}
\label{as:E1}
  \big\|
    E_n^{h,k}x
  \big\|
& \leq
  K_{\epsilon}
  \big(
    h^{\sigma}
    +
    k^{\frac{\sigma}2}
  \big)
  \|x\|_{\dot{H}^{\sigma(1+\epsilon)}},
  \quad
  0\leq\sigma\leq 2
  ,\
  \epsilon>0,
\end{align}
and the non-smooth data error estimates, for $n=1,\dots,N$, $t>0$,
\begin{align}
\label{as:E2}
& \big\|
    A^{\frac{s}{2\rho}}
    E_n^{h,k}
    x
  \big\|
  \leq
  R_s
  \big(
    h^{\frac\sigma\rho}
    +
    k^{\frac{\sigma}2}
  \big)
  t_n^{-\frac{\sigma+s}{2}}
  \|x\|,
  \quad
  0\leq\sigma\leq 2,\
  0\leq s\leq1-\sigma/2,\\
\label{as:e}
& \big\|
    \big(
      e^{-tA} - e^{-tA_h} P_h
    \big)
    x
  \big\|
  \leq
  R_0 h^{\sigma}
  t^{-\frac\sigma2}
  \|x\|
  ,\quad
  0\leq\sigma\leq2,
\end{align}
where $(e^{-tA})_{t\in[0,\infty)}$ and $(e^{-tA_h})_{t\in[0,\infty)}$ are the analytic semigroups generated by $-A$ and $-A_h$, respectively. We introduce the piecewise continuous operator function $\tilde{E}^{h,k}\colon[0,T]\to \LB$ given by $\tilde{E}_t^{h,k}=S_t-B_n^{h,k}$ for $t\in[t_n,t_{n+1})$ and $n=0,\dots N-1$. By \eqref{as:S} and \eqref{as:E1} the family $(\tilde{E}_t^{h,k})_{t\in[0,T]}$ satisfies for $t\in(0,T]$ the bound
\begin{align}
\label{ineq:Etilde}
  \big\|
    A^{\frac{s}{2\rho}}
    \tilde{E}_t^{h,k}
  \big\|_{\LB}
& \leq
  R_s
  \big(
    h^{\frac\sigma\rho}
    +
    k^{\frac{\sigma}2}
  \big)
  t^{-\frac{\sigma+s}{2}}
  ,\quad
  0\leq\sigma\leq 2,\
  0\leq s\leq1-\sigma/2.
\end{align}
The discrete and continuous stochastic convolutions are defined by
\begin{align*}
  W_t^S=\int_0^t S_{t-s}\diff{W_s}
  ,\quad
  t\in[0,T]
  ;\quad
  W_n^{B^{h,k}}
  =
  \sum_{j=0}^{n-1}
    \int_{t_j}^{t_{j+1}}
      B_{n-j}^{h,k}
    \diff{W_t},
  \quad
  n=1,\dots,N.
\end{align*}

We now define approximations of equation \eqref{eq:SVIE}. For $h,k\in(0,1)$, let $(X_n^{h,k})_{n=0}^N$ be the solution to the equation
\begin{align}
\label{eq:approx}
  X_n^{h,k}
  =
  B_n^{h,k} x_0
  +
  k\sum_{j=1}^{n-1}
    B_{n-j}^{h,k}
    F(X_j^{h,k})
  +
  W_n^{B^{h,k}},
  \quad
  n=1,\dots,N.
\end{align}

\subsection{Strong convergence}
Boundedness in the $L^p(\Omega;H)$-sense of the approximate family $(X_n^{h,k})_{n=0}^N$ is stated in the next proposition. For a proof in the parabolic case, i.e., for $\rho=1$, see \cite{AnderssonKruseLarsson}*{Proposition 3.15}. The general case is proved in the same way but using the different smoothing property in \eqref{as:B}.
\begin{proposition}
\label{lemma:Stability}
Let the setting of Section \ref{sec4:1} hold. For $p\geq 2$ it holds that
\begin{align*}
  \sup_{h,k\in(0,1)}
  \max_{n\in\{0,\dots,N\}}
  \big\|
    X_n^{h,k}
  \big\|_{L^p(\Omega;H)}
  <\infty.
\end{align*}
\end{proposition}

We next prove strong convergence. This is interesting in itself, but it is also used in our proof of the Malliavin regularity of $X$ in Proposition \ref{lemma:Stability3}.
\begin{theorem}
\label{thm:strong}
Let the setting of Section \ref{sec4:1} hold, let $X$ be the solution to \eqref{eq:SVIE} and let $(X^{h,k})_{h,k\in(0,1]}$ be the solutions to \eqref{eq:approx}. For $\gamma\in[0,\beta)$, $p\in[2,\infty)$, there exists $C>0$ such that
\begin{align*}
& \max_{n\in\{0,\dots,N\}}
  \big\|
    X_{t_n}-X_n^{h,k}
  \big\|_{L^p(\Omega;H)}
  \leq
  C
  \big(
    h^{\gamma}+k^{\frac{\rho\gamma}2}
  \big)
  ,\quad
  h,k\in(0,1).
\end{align*}
\end{theorem}

\begin{proof}
We take the difference of \eqref{eq:SVIE} and \eqref{eq:approx} to obtain the equation for the error,
\begin{equation}
\begin{split}
\label{eq:difference}
  X_{t_n}-X_n^{h,k}
& =
  \big(
    S_{t_n}-B_n^{h,k}
  \big)
  x_0
  +
  \sum_{j=0}^{n-1}
    \int_{t_j}^{t_{j+1}}
      \big(
        S_{t_n-t}-B_{n-j}^{h,k}
      \big)
      F(X_t)
    \diff{t}\\
& \quad
  +
  \sum_{j=0}^{n-1}
    \int_{t_j}^{t_{j+1}}
      B_{n-j}^{h,k}
    \big(
      F(X_{t})-F(X_j^{h,k})
    \big)
    \diff{t}
  +
  W_{t_n}^S-W_n^{B^{h,k}}.
\end{split}
\end{equation}
The deterministic nature of the first two terms allows us to obtain twice the rate of convergence compared to the other terms.  This will be used later in the proof of Lemma~\ref{thm:strong2}. Recall that $\tilde{E}_t^{h,k}=S_t-B_{n}^{h,k}$ for $t\in[t_n,t_{n+1})$ and $n=0,\dots,N-1$. We get
\begin{align*}
  \big\|
    X_{t_n}-X_n^{h,k}
  \big\|_{L^p(\Omega;H)}
& \leq
  \big\|
    E_n^{h,k}
    x_0
  \big\|_H
  +
  \Big\|
    \int_0^{t_n}
      \tilde{E}_{t_n-t}^{h,k}
      F(X_t)
    \diff{t}
  \Big\|_{L^p(\Omega;H)}\\
& \quad
  +
  \Big\|
    \sum_{j=0}^{n-1}
      \int_{t_j}^{t_{j+1}}
        B_{n-j}^{h,k}
        \big(
          F(X_{t})-F(X_j^{h,k})
        \big)
      \diff{t}
  \Big\|_{L^p(\Omega;H)}\\
& \quad
  +
  \big\|
    W_{t_n}^S-W_n^{B^{h,k}}
  \big\|_{L^p(\Omega;H)}.
\end{align*}
Using \eqref{as:x0}, \eqref{as:E1} with $\sigma=2\rho\gamma$, $\epsilon=(3-2\gamma\rho)/2\gamma\rho$ we obtain
\begin{align}
\label{eq:x0}
  \max_{n\in\{0,\dots,N\}}
  \big\|
    E_n^{h,k}x_0
  \big\|
  \leq
  K_{\frac{3-2\gamma\rho}{2\gamma\rho}}
  \big(
    h^{2\rho\gamma}+k^{\rho\gamma}
  \big)
  \|x_0\|_{\dot{H}^{3}}.
\end{align}
By Proposition \ref{prop:existence}, \eqref{as:F}, \eqref{ineq:Etilde} it holds that
\begin{equation}
\label{eq:Fconv}
\begin{split}
& \Big\|
    \int_0^{t_n}
      \tilde{E}_{t_n-t}^{h,k}
      F(X_t)
    \diff{t}
  \Big\|_{L^p(\Omega;H)}
  \leq
  \int_0^{t_n}
    \big\|
      \tilde{E}_{t_n-t}^{h,k}
    \big\|_\LB
    \big\|
      F(X_t)
    \big\|_{L^p(\Omega;H)}
  \diff{t}\\
& \qquad
  \leq
  R_0
  \big(
    h^{2\gamma}
    +
    k^{\rho\gamma}
  \big)
  |F|_{\Cb^1(H;H)}
  \Big(
    1+\sup_{t\in[0,T]}\big\|X_t\big\|_{L^p(\Omega;H)}
  \Big)
  \int_0^{t_n}
    (t_n-t)^{-\rho\gamma}
  \diff{t}\\
& \qquad
  \lesssim
    h^{2\gamma}
    +
    k^{\rho\gamma}.
\end{split}
\end{equation}
Using \eqref{as:F}, \eqref{eq:Taylor}, \eqref{as:B}, and Proposition \ref{prop:Holder} yields
\begin{align*}
& \Big\|
    \sum_{j=0}^{n-1}
      \int_{t_j}^{t_{j+1}}
        B_{n-j}^{h,k}
        \big(
          F(X_t)-F(X_j^{h,k})
        \big)
      \diff{t}
  \Big\|_{L^p(\Omega;H)}\\
& \quad
  \leq
  |F|_{\Cb^1(H;H)}
  \sum_{j=0}^{n-1}
    \int_{t_j}^{t_{j+1}}
      \big\|
        B_{n-j}^{h,k}
      \big\|_{\LB}
      \big\|
        X_t-X_j^{h,k}
      \big\|_{L^p(\Omega;H)}
    \diff{t}\\
& \quad
  \leq
  L_0
  |F|_{\Cb^1(H;H)}
  \sum_{j=0}^{n-1}
    \int_{t_j}^{t_{j+1}}
      \Big(
        \big\|
          X_t-X_{t_j}
        \big\|_{L^p(\Omega;H)}
        +
        \big\|
          X_{t_j}-X_j^{h,k}
        \big\|_{L^p(\Omega;H)}
      \Big)
    \diff{t}\\
& \quad
  \leq
  L_0
  |F|_{\Cb^1(H;H)}
  \Big(
    CTk^{\frac{\rho\gamma}2}
    +
    k
    \sum_{j=0}^{n-1}
      \big\|
        X_{t_j}-X_j^{h,k}
      \big\|_{L^p(\Omega;H)}
  \Big).
\end{align*}
For the error of the stochastic convolution
we write the difference in the form
\begin{align}
\label{eq:diffW}
  W_{t_n}^S-W_n^{B^{h,k}}
& =
  \sum_{j=0}^{n-1}
    \int_{t_j}^{t_{j+1}}
      \big(
        S_{t_n-t}-B_{n-j}^{h,k}
      \big)
    \diff{W_t}\\
& =
  \int_0^{t_n}
    \tilde{E}_{t_n-t}^{h,k}
  \diff{W_t}
  =
  \int_0^{t_n}
    \tilde{E}_{t}^{h,k}
  \diff{W_t}.
\end{align}
By \eqref{ineq:Burkholder} and \eqref{ineq:Etilde} with $\sigma=\gamma\rho$, and $s=1-\beta\rho$, we obtain the estimate
\begin{align*}
& \big\|
    W_{t_n}^S-W_n^{B^{h,k}}
  \big\|_{L^p(\Omega;H)}
  \leq
  \Big(
    \frac{p(p-1)}2
    \int_{0}^{t_n}
      \big\|
        A^{\frac{\beta\rho-1}{2\rho}}
      \big\|_{\LB_2^0}^{2}
      \big\|
        A^{\frac{1-\beta\rho}{2\rho}}
        \tilde{E}_t^{h,k}
      \big\|_{\LB}^{2}
    \diff{t}
  \Big)^\frac12\\
& \quad
  \lesssim
  R_{1-\beta\rho}
  \Big(
    \int_{0}^{t_n}
      t^{\rho(\beta-\gamma)-1}
    \diff{t}
  \Big)^\frac12
  \big(
    h^{\gamma}
    +
    k^{\frac{\rho\gamma}2}
  \big)
  \lesssim
  h^{\gamma}
  +
  k^{\frac{\rho\gamma}2}.
\end{align*}
Collecting the estimates yields that, for all $n=0,\dots,N$, it holds
\begin{align*}
  \big\|
    X_{t_n}-X_n^{h,k}
  \big\|_{L^p(\Omega;H)}
  \lesssim
  h^\gamma+k^{\frac{\rho\gamma}2}
  +
  k\sum_{j=0}^{n-1}
    \big\|
      X_{t_j}-X_j^{h,k}
    \big\|_{L^p(\Omega;H)}.
\end{align*}
The proof is completed by Gronwall's lemma.
\end{proof}

\subsection{Regularity and weak convergence}
Here we state and prove our main result on weak convergence. It is based on a strong error estimate in the $\mathbf{M}^{1,p}(H)^*$ norm combined with boundedness of $X$ and $X^{h,k}$ in $\mathbf{M}^{1,p,q}(H)$ for suitable $p,q$. The methodology was introduced in \cite{AnderssonKruseLarsson}, but here we exploit it further in a more general setting. We begin by proving the Malliavin differentiability of $X^{h,k}$.

\begin{proposition}
\label{lemma:Stability2}
Let the setting of Section \ref{sec4:1} hold, and let $X^{h,k}$ be the
solution to \eqref{eq:approx}. For $p\in[2,\infty)$,
$q\in[2,\frac2{1-\rho\beta})$, it holds that 
\begin{align*}
  \sup_{h,k\in(0,1)}
  \max_{n\in\{0,\dots,N\}}
  \Big(
  \big\|
    X_n^{h,k}
  \big\|_{\mathbf{M}^{1,p,q}(H)}
    +\big|
      X_n^{h,k}
    \big|_{\mathbf{M}^{1,\infty,q}(H)}
  \Big)
  <\infty.
\end{align*}
\end{proposition}

\begin{proof}[Sketch of proof]
Note first that $DX_0^{h,k}=0$ as $X_0^{h,k}$ is deterministic. Therefore it follows inductively that $X_j^{h,k}$, $j=0,\dots,N$, are differentiable and the derivative satisfies the equation
\begin{align}
\label{eq:DXhk}
  D_rX_n^{h,k}
  =
  k \sum_{j=0}^{n-1}
    B_{n-j}^{h,k}
    F'(X_j^{h,k})
    D_rX_j^{h,k}
  +
  \sum_{j=0}^{n-1}
    \chi_{[t_j,t_{j+1})}(r)
    B_{n-j}^{h,k}.
\end{align}
The proof is performed by straightforward analysis of this equation
using the discrete Gronwall's lemma, see
\cite{AnderssonKruseLarsson}*{Proposition 3.16} for details in the
parabolic case $\rho=1$. The general case is treated analogously.
\end{proof}

The Malliavin regularity of $X$ is next obtained by a limiting procedure.
\begin{proposition}
\label{lemma:Stability3}
Let the setting of Section \ref{sec4:1} hold and let $X$ be the solution to \eqref{eq:SVIE}. For $p\in[2,\infty)$, $q\in[2,\frac2{1-\rho\beta})$, it holds that $X\in \C(0,T;\mathbf{M}^{1,p,q}(H))$, and moreover it holds that
\begin{align*}
  \sup_{t\in[0,T]}
  \big|
    X_t
  \big|_{\mathbf{M}^{1,\infty,q}(H)}
  <\infty.
\end{align*}
\end{proposition}

\begin{proof}
Let $\tilde{X}_t^{h,k}=X_n^{h,k}$ for $t\in[t_n,t_{n+1})$, $n=0,\dots,N-1$, $h,k\in(0,1)$. By Proposition \ref{lemma:Stability2} it holds in particular, that the family $(\tilde{X}^{h,k})_{h,k\in(0,1)}$ is bounded in the Hilbert space $\mathcal{X}=L^2(0,T;\mathbf{M}^{1,2,2}(H))$, and by Theorem \ref{thm:strong} it holds that $\tilde{X}^{h,k}\to X$ as $h,k\to 0$ in the Hilbert space $\mathcal{Y}=L^2(0,T;L^2(\Omega;H))$. Lemma \ref{lemma:LimitReg} applies and ensures that $X\in\mathcal{X}=L^2(0,T;\mathbf{M}^{1,2,2}(H))$.

By \cite{FuhrmanTessitore}*{Lemma 3.6} it holds that also
$ \int_0^\cdot S_{\cdot-s}F(X_s)\diff{s}\in
L^2(0,T;\mathbf{M}^{1,2,2}(H))$
with
$D_r\int_0^tS_{t-s}F(X_s)\diff{s}=\int_r^tS_{t-s}F'(X_s)D_rX_s\diff{s}$,
for $0\leq r\leq t\leq T$, and
$\int_0^\cdot S_{\cdot-s}\diff{W_s}\in L^2(0,T;\mathbf{M}^{1,2,2}(H))$
with $D_r\int_0^tS_{t-s}\diff{W_s}=S_{t-r}$, for
$0\leq r\leq t\leq T$. We remark that \cite{FuhrmanTessitore}*{Lemma 3.6} is
formulated for semigroups, but the semigroup property is not used in
the proof. We have thus proved that we can differentiate the equation
for $X$ term by term, and obtain the equation
\begin{align*}
  D_rX_t
  =
  \begin{cases}
    S_{t-r}
    +
    \int_r^t
      S_{t-s}
      F'(X_s)
      D_rX_s
    \diff{s}
    ,\quad
    &t\in(r,T],\\
    0,
    &t\in[0,r].
  \end{cases}
\end{align*}
A straightforward analysis of this equation, by a Gronwall argument,
remove as in the proof of \cite{AnderssonKruseLarsson}*{Proposition
  3.10} completes the proof.
\end{proof}

In the proof of \cite{AnderssonKruseLarsson}*{Lemma 4.6}, which is the analogue of Lemma \ref{thm:strong2} below, a bound
\begin{align}
\label{eq:equiv_norms}
  \|A_h^{-\frac\delta2}P_hx\|
  \leq
  \|A_h^{\frac\delta2}P_hA^{-\frac\delta2}\|_{\LB}
  \|A^{-\frac\delta2}x\|
  \leq
  C
  \|A^{-\frac\delta2}x\|,
\end{align}
was used in the special case $\delta=1$. This estimate is true for all
$\delta\in[0,1]$ for both the finite element method and for spectral
approximation. For $\delta>1$ it holds only for spectral
approximation. In this paper we need $\delta\in[0,2/\rho)$ and
therefore we cannot rely on \eqref{eq:equiv_norms}. In
\cite{thomee2006}*{Lemma 5.3} it is shown that for finite element
discretization and for $\delta=0,1,2$ it holds
\begin{align*}
  \|A_h^{-\frac\delta2}P_hx\|
  \leq
  C
  \big(
    \|A^{-\frac\delta2}x\|
    +h^\delta\|x\|
  \big)
  ,\quad
  x\in H.
\end{align*}
The next lemma is a generalization of this result, assuming the availability of a non-smooth data estimate of the form \eqref{as:e}. It will be used in the proof of Lemma~\ref{thm:strong2} below with $\mathcal{X}=\mathbf{M}^{1,p}(H)^*$ for a certain $p$. By using it we need not rely on \eqref{eq:equiv_norms} and in this way we include finite element discretization under the same generality as spectral approximations.
\begin{lemma}
\label{lemma:ChangOfNorm}
Let the setting of Section \ref{sec4:1} hold and let $\mathcal{X}$ be a Banach space such that the embedding $L^2(\Omega;H)\subset \mathcal{X}$ is continuous. For $\kappa\in [0,2)$, $\sigma\in[0,\kappa)$, there exists $C>0$ such that for $Y\in L^2(\Omega;H)$ it holds that
\begin{align*}
  \big\|
    A_h^{-\frac\kappa2}P_h
    Y
  \big\|_{\mathcal{X}}
& \leq
  \big\|
    A^{-\frac\kappa2}Y
  \big\|_{\mathcal{X}}
  +
  C h^\sigma
  \big\|
    Y
  \big\|_{L^2(\Omega;H)}
  ,\quad
  h\in(0,1).
\end{align*}
\end{lemma}

\begin{proof}
By the continuous embedding $L^2(\Omega;H)\subset \mathcal{X}$ we get that
\begin{align*}
  \big\|
    A_h^{-\frac\kappa2}P_h
    Y
  \big\|_{\mathcal{X}}
& \leq
  \big\|
    A^{-\frac\kappa2}
    Y
  \big\|_{\mathcal{X}}
  +
  \big\|
    \big(
      A_h^{-\frac\kappa2}P_h-A^{-\frac\kappa2}
    \big)
    Y
  \big\|_{\mathcal{X}}\\
& \lesssim
  \big\|
    A^{-\frac\kappa2}
    Y
  \big\|_{\mathcal{X}}
  +
  \big\|
    A_h^{-\frac\kappa2}P_h-A^{-\frac\kappa2}
  \big\|_{\LB}
  \big\|
    Y
  \big\|_{L^2(\Omega;H)}.
\end{align*}
By \cite{pazy1983}*{Chapter 2, (6.9)} we have
\begin{align*}
  A_h^{-\frac\kappa2}P_h
  -
  A^{-\frac\kappa2}
  =
  \frac1{\Gamma(\kappa/2)}
  \int_0^\infty
    t^{\frac\kappa2-1}
    \big(
      e^{-tA_h}P_h - e^{-tA}
    \big)
  \diff{t}.
\end{align*}
Therefore, by \eqref{as:e},
\begin{align*}
& \big\|
    A_h^{-\frac\kappa2}P_h-A^{-\frac\kappa2}
  \big\|_{\LB}
  \leq
  \frac1{\Gamma(\kappa/2)}
  \int_0^\infty
    t^{\frac\kappa2-1}
    \big\|
      e^{-tA_h}P_h-e^{-tA}
    \big\|_{\LB}
  \diff{t}\\
& \quad
  \lesssim
  \int_0^{h^{-2}}
    t^{\frac\kappa2-1}
    \big\|
      e^{-tA_h}P_h-e^{-tA}
    \big\|_{\LB}
  \diff{t}
  +
  \int_{h^{-2}}^\infty
    t^{\frac\kappa2-1}
    \big\|
      e^{-tA_h}P_h-e^{-tA}
    \big\|_{\LB}
  \diff{t}\\
& \quad
  \lesssim
  h^{\frac{\kappa+\sigma}2}
  \int_0^{h^{-2}}
    t^{\frac{\kappa-\sigma}4-1}
  \diff{t}
  +
  h^2\int_{h^{-2}}^\infty
    t^{\frac\kappa2-2}
  \diff{t}
  =
  \frac{
    4h^{\sigma}
  }
  {
    \kappa-\sigma
  }
  +
  \frac{
    2
  }
  {
    2-\kappa
  }
  h^2
  h^{2-\frac\kappa2}
  \,
  \lesssim
  \,
  h^\sigma.
\end{align*}
\end{proof}

The next result is a strong error estimate in the $\mathbf{M}^{1,p}(H)^*$ norm. Together with the regularity stated in Propositions~\ref{lemma:Stability2} and \ref{lemma:Stability3} it is the key to the proof of Theorem~\ref{thm:main} below on weak convergence.

\begin{lemma}
\label{thm:strong2}
Let the setting of Section \ref{sec4:1} hold, and let $X$ and $X^{h,k}$ be the solutions to \eqref{eq:SVIE} and \eqref{eq:approx}, respectively. For $\gamma\in[0,\beta)$, $p=\frac2{1-\rho\gamma}$, there exists $C>0$ such that
\begin{align*}
  \max_{n\in\{0,\dots,N\}}
  \big\|
    X_{t_n}-X_n^{h,k}
  \big\|_{\mathbf{M}^{1,p}(H)^*}
  \leq
  C
  \big(
    h^{2\gamma}+k^{\rho\gamma}
  \big)
  ,\quad
  h,k\in(0,1).
\end{align*}
\end{lemma}

\begin{proof}
The proof is performed essentially as that of Theorem \ref{thm:strong}. By \eqref{eq:difference} and the continuous embeddings
$
  H
  \subset
  L^p(\Omega;H)
  \subset
  L^2(\Omega;H)
  \subset
  \mathbf{M}^{1,p}(H)^*
$,
it follows that
\begin{align*}
  \big\|
    X_{t_n}-X_n^{h,k}
  \big\|_{\mathbf{M}^{1,p}(H)^*}
& \leq
  \big\|
    E_n^{h,k}
    x_0
  \big\|_H
  +
  \Big\|
    \int_0^{t_n}
      \tilde{E}_{t_n-t}^{h,k}
      F(X_t)
    \diff{t}
  \Big\|_{L^p(\Omega;H)}\\
& \quad
  +
  \Big\|
    \sum_{j=0}^{n-1}
      \int_{t_j}^{t_{j+1}}
        B_{n-j}^{h,k}
        \big(
          F(X_{t})-F(X_j^{h,k})
        \big)
      \diff{t}
  \Big\|_{\mathbf{M}^{1,p}(H)^*}\\
& \quad
  +
  \big\|
    W_{t_n}^S-W_n^{B^{h,k}}
  \big\|_{\mathbf{M}^{1,p}(H)^*}.
\end{align*}
The first two terms was already estimated as desired in \eqref{eq:x0} and \eqref{eq:Fconv}. Choose $\kappa$ so that $\max(\delta,2\gamma)<\kappa<2/\rho$, where $\delta$ is the parameter in \eqref{as:F}. Since $\rho\kappa<2$, we have, by Lemma \ref{lemma:operator_norm} and \eqref{as:B} with $s=\rho\kappa/2$, that
\begin{align*}
& \Big\|
    \sum_{j=0}^{n-1}
      \int_{t_j}^{t_{j+1}}
        B_{n-j}^{h,k}
        \big(
          F(X_t)-F(X_j^{h,k})
        \big)
      \diff{t}
  \Big\|_{\mathbf{M}^{1,p}(H)^*}\\
& \quad
  \leq
  \sum_{j=0}^{n-1}
    \int_{t_j}^{t_{j+1}}
      \big\|
        B_{n-j}^{h,k}
        A_h^{\frac\kappa2}
        P_h
      \big\|_{\LB}
      \big\|
        A_h^{-\frac\kappa2}
        P_h
        \big(
          F(X_t)-F(X_j^{h,k})
        \big)
      \big\|_{\mathbf{M}^{1,p}(H)^*}
    \diff{t}\\
& \quad
  \leq
  L_{\frac{\kappa\rho}2}
  \sum_{j=0}^{n-1}
    \int_{t_j}^{t_{j+1}}
      t_{n-j}^{-\frac{\kappa\rho}2}
      \big\|
        A_h^{-\frac\kappa2}
        P_h
        \big(
          F(X_t)-F(X_j^{h,k})
        \big)
      \big\|_{\mathbf{M}^{1,p}(H)^*}
    \diff{t}.
\end{align*}
Applying Lemma \ref{lemma:ChangOfNorm} with $\mathcal{X}=\mathbf{M}^{1,p}(H)^*$ and $\sigma=2\gamma<\kappa$ yields
\begin{align*}
& \big\|
    A_h^{-\frac\kappa2}P_h
    \big(
      F(X_t)
      -
      F(X_j^{h,k})
    \big)
  \big\|_{\mathbf{M}^{1,p}(H)^*}\\
& \quad
  \leq
  C
  h^{2\gamma}
  \big\|
    F(X_t)-F(X_j^{h,k})
  \big\|_{L^2(\Omega;H)}
  +
  \big\|
    A^{-\frac\kappa2}
    \big(
      F(X_t)
      -
      F(X_j^{h,k})
    \big)
  \big\|_{\mathbf{M}^{1,p}(H)^*}.
\end{align*}
For the first term we get by \eqref{as:F}, Propositions \ref{prop:existence}, and \ref{lemma:Stability} that
\begin{align*}
& \sup_{t\in[0,T]}
  \max_{j\in\{0,\dots,N\}}
  \big\|
    F(X_t)-F(X_j^{h,k})
  \big\|_{L^2(\Omega;H)}\\
& \quad
  \leq
  |F|_{\Cb^1(H;H)}
  \Big(
    \sup_{t\in[0,T]}
    \big\|
      X_t
    \big\|_{L^2(\Omega;H)}
    +
    \max_{j\in\{0,\dots,N\}}
    \big\|
      X_j^{h,k}
    \big\|_{L^2(\Omega;H)}
  \Big)
  <\infty.
\end{align*}
By duality in the Gelfand triple $\mathbf{M}^{1,p}(\dot{H}^{-\delta})\subset L^2(\Omega;\dot{H}^{-\delta})\subset \mathbf{M}^{1,p}(\dot{H}^{-\delta})^*$ we compute that for $Y\in L^2(\Omega;\dot{H}^{-\delta})$,
\begin{align*}
  \|Y\|_{\adam{\mathbf{M}^{1,p}(\dot{H}^{-\delta})^*}}
& =
  \sup_{Z\in\adam{\mathbf{M}^{1,p}(\dot{H}^{-\delta})}}
  \frac{
    \big\langle
      Z,Y
    \big\rangle_{L^2(\Omega;\dot{H}^{-\delta})}
  }{
    \|Z\|_{\mathbf{M}^{1,p}(\dot{H}^{-\delta})}
  }\\
& =
  \sup_{Z\in\mathbf{M}^{1,p}(\dot{H}^{-\delta})}
  \frac{
    \big\langle
      A^{-\frac\delta2}Z,A^{-\frac\delta2}Y
    \big\rangle_{L^2(\Omega;H)}
  }{
    \|Z\|_{\mathbf{M}^{1,p}(\dot{H}^{-\delta})}
  }\\
& =
  \sup_{Z\in\mathbf{M}^{1,p}(\dot{H}^{-\delta})}
  \frac{
    \big\langle
      Z,A^{-\frac\delta2}Y
    \big\rangle_{L^2(\Omega;H)}
  }{
    \|A^{\frac\delta2}Z\|_{\mathbf{M}^{1,p}(\dot{H}^{-\delta})}
  }\\
& =
  \sup_{Z\in\mathbf{M}^{1,p}(H)}
  \frac{
    \big\langle
      Z,A^{-\frac\delta2}Y
    \big\rangle
  }{
    \|Z\|_{\mathbf{M}^{1,p}(H)}
  }
  =
  \|A^{-\frac\delta2}Y\|_{\mathbf{M}^{1,p}(H)^*}.
\end{align*}
Therefore, by Lemma \ref{lemma:operator_norm} and Lemma \ref{lemma:Lipschitz} applied with $U=H$, $V=\dot{H}^{-\delta}$, $\sigma=F$ we get
\begin{align*}
& \big\|
    A^{-\frac\kappa2}
    \big(
      F(X_t)
      -
      F(X_j^{h,k})
    \big)
  \big\|_{\mathbf{M}^{1,p}(H)^*}\\
& \quad
  \leq
  \big\|
    A^{-\frac{\kappa-\delta}2}
  \big\|_{\LB}
  \big\|
    A^{-\frac\delta2}
    \big(
      F(X_t)
      -
      F(X_j^{h,k})
    \big)
  \big\|_{\mathbf{M}^{1,p}(H)^*}\\
& \quad
  =
  \big\|
    A^{-\frac{\kappa-\delta}2}
  \big\|_{\LB}
  \big\|
    F(X_t)
    -
    F(X_j^{h,k})
  \big\|_{\mathbf{M}^{1,p}(\dot{H}^{-\delta})^*}\\
& \quad
  \leq
  \big\|
    A^{-\frac{\kappa-\delta}2}
  \big\|_{\LB}
  \max
  \Big(
    |F|_{\Cb^1(H;\dot{H}^{-\delta})}
    ,
    |F|_{\Cb^2(H;\dot{H}^{-\delta})}
  \Big)\\
& \qquad
  \times
  \Big(
    \sup_{j\in\{0,\dots,N\}}
    \big|
      X_j^{h,k}
    \big|_{\mathbf{M}^{1,\infty,p}(H)}
    +
    \sup_{t\in[0,T]}
    \big|
      X_t
    \big|_{\mathbf{M}^{1,\infty,p}(H)}
  \Big)
\\
& \qquad
    \times
    \Big(
      \big\|
        X_t-X_{t_j}
      \big\|_{\mathbf{M}^{1,p}(H)^*}
      +
      \big\|
        X_{t_j}-X_j^{h,k}
      \big\|_{\mathbf{M}^{1,p}(H)^*}
    \Big).
  \end{align*}
By Propositions \ref{prop:Holder} and \ref{lemma:Stability2} and Proposition \ref{lemma:Stability3}, we conclude
\begin{align*}
& \big\|
    A^{-\frac\kappa2}
    \big(
      F(X_t)
      -
      F(X_j^{h,k})
    \big)
  \big\|_{\mathbf{M}^{1,p}(H)^*}
  \lesssim
  k^{\rho\gamma}
  +
  \big\|
    X_{t_j}-X_j^{h,k}
  \big\|_{\mathbf{M}^{1,p}(H)^*}.
\end{align*}
Thus,
\begin{align*}
& \Big\|
    \sum_{j=0}^{n-1}
      \int_{t_j}^{t_{j+1}}
        B_{n-j}^{h,k}
        \big(
          F(X_t)-F(X_j^{h,k})
        \big)
      \diff{t}
  \Big\|_{\mathbf{M}^{1,p}(H)^*}\\
& \quad
  \lesssim
  h^{2\gamma} + k^{\rho\gamma}
  +
  k
  \sum_{j=0}^{n-1}
    t_{n-j}^{-\frac{\kappa\rho}2}
    \big\|
      X_{t_j}-X_j^{h,k}
    \big\|_{\mathbf{M}^{1,p}(\dot{H}^{-\delta})^*}.
\end{align*}

By \eqref{eq:diffW}, \eqref{ineq:BurkholderDual}, and \eqref{ineq:Etilde}, with $s=1-\beta\rho$, $\sigma=2\gamma\rho$, and since $p=\tfrac2{1-\rho\gamma}$ and $p'=\tfrac2{1+\rho\gamma}$, we get
\begin{align*}
& \big\|
    W_{t_n}^S-W_n^{B^{h,k}}
  \big\|_{\mathbf{M}^{1,p}(H)^*}\\
& \quad
  \leq
  \Big(
    \int_{0}^{t_n}
      \big\|
        A^{\frac{\beta\rho-1}{2\rho}}
      \big\|_{\LB_2^0}^{\frac2{1+\rho\gamma}}
      \big\|
        A^{\frac{1-\beta\rho}{2\rho}}
        \tilde{E}_t^{h,k}
      \big\|_{\LB}^{\frac2{1+\rho\gamma}}
    \diff{t}
    \Big)^{\frac{1+\rho\gamma}2}\\
& \quad
  \leq
  R_{1-\beta\rho}
  \big\|
    A^{\frac{\beta\rho-1}{2\rho}}
  \big\|_{\LB_2^0}
  \Big(
    \int_{0}^{t_n}
      t^{\frac{\rho(\beta-\gamma)}{1+\rho\gamma}-1}
    \diff{t}
    \Big)^\frac{1+\rho\gamma}2
  \big(
    h^{2\gamma}
    +
    k^{\rho\gamma}
  \big).
\end{align*}
Altogether we have that for every $n=1,\dots,N$ it holds that
\begin{align*}
  \big\|
    X_{t_n}-X_n^{h,k}
  \big\|_{\mathbf{M}^{1,p}(H)^*}
  \lesssim
  h^{2\gamma}+k^{\rho\gamma}
  +
  k\sum_{j=0}^{n-1}
    t_{n-j}^{-\frac{\kappa\rho}2}
    \big\|
      X_{t_j}-X_j^{h,k}
    \big\|_{\mathbf{M}^{1,p}(H)^*}.
\end{align*}
Lemma \ref{lemma2:Gronwall} finishes the proof.
\end{proof}

We next state our main result on weak convergence. We remark that to
the best of our knowledge all previous weak convergence results concern
convergence of $| \E [\varphi(X_\tau^{h,k})-\varphi(X_\tau)]|$ for
fixed $\tau\in [0,T]$, which is a special case of the following
theorem.

\begin{theorem}
\label{thm:main}
Let $X$ and $X^{h,k}$ be the solutions to \eqref{eq:SVIE} and \eqref{eq:approx}, respectively. Let $\tilde{X}_t^{h,k}=X_n^{h,k}$, for $t\in[t_n,t_{n+1})$, $n\in\{0,\dots,N-1\}$ and $\tilde{X}_t^{h,k}=X_N^{h,k}$, for $t\in[t_N,T]$. For $K\geq1$, $m_1,\dots,m_K\geq2$, $\varphi_i\in\Cp^{2,m_i}(H;\R)$, $\nu_i\in\mathcal{M}_T$, $i=1,\dots,K$, $\Phi(Z)=\prod_{i=1}^K\varphi_i(\int_0^T Z_t\diff{\nu_{i,t}})$, $\gamma\in[0,\beta)$, there exists $C>0$ such that
\begin{align*}
  \big|
  \E
  \big[
    \Phi(X)
    -
    \Phi(\tilde{X}^{h,k})
  \big]
  \big|
  \leq
  C
  \big(
    h^{2\gamma}+k^{\rho\gamma}
  \big)
  ,\quad
  h,k\in(0,1).
\end{align*}
\end{theorem}

\begin{proof}
We start by observing that by \eqref{eq:Taylor} we have
\begin{align*}
& \prod_{i=1}^K
  \varphi_i(x_i)
  -
  \prod_{i=1}^K
  \varphi_i(y_i)\\
& \quad
  =
  \sum_{l=1}^K
  \prod_{i=1}^{l-1}
  \varphi_i(x_i)
  \prod_{j=l+1}^{K}
  \varphi_j(y_j)
   \big(
      \varphi_l(x_l)
      -
      \varphi_l(y_l)
    \big)
\\
& \quad
  =
  \sum_{l=1}^K
  \Bigg\langle
    \prod_{i=1}^{l-1}
    \varphi_i(x_i)
    \prod_{j=l+1}^{K}
    \varphi_j(y_j)
    \int_0^1
      \varphi_l'(y_l+\lambda(x_l-y_l))
    \diff{\lambda}
    ,
    x_l-y_l
  \Bigg\rangle\\
&
  \quad
  =:
  \sum_{l=1}^K
  \langle
    \gamma_l(x_1,\dots,x_{l}, y_l,\dots,y_K)
    ,
    x_l-y_l
  \rangle.
\end{align*}
Here we use the convention that an empty product equals 1.
We get
\begin{align*}
& \big|
    \E
    \big[
      \Phi(X)
      -
      \Phi(\tilde{X}^{h,k})
    \big]
  \big|
  =
  \Big|
  \sum_{l=1}^K
  \Big\langle
    \gamma_l(Y_l^{h,k})
    ,
    \int_0^T
      \big(
        X_t-\tilde{X}_t^{h,k}
      \big)
    \diff{\nu_{l,t}}
  \Big\rangle_{L^2(\Omega;H)}
  \Big|,
\end{align*}
where
\begin{align*}
  Y_l^{h,k}=
  \Big(
    \int_0^T
      X_t
    \diff{\nu_{1,t}}
    ,\dots,
    \int_0^T
      X_t
    \diff{\nu_{l,t}}
    ,
    \int_0^T
      \tilde{X}_t^{h,k}
    \diff{\nu_{l,t}}
    ,\dots,
    \int_0^T
      \tilde{X}_t^{h,k}
    \diff{\nu_{K,t}}
  \Big).
\end{align*}
By duality in the Gelfand triple $\mathbf{M}^{1,p}(H)\subset L^2(\Omega;H)\subset\mathbf{M}^{1,p}(H)^*$ we obtain
\begin{align*}
& \big|
  \E
  \big[
    \Phi(X)
    -
    \Phi(\tilde{X}^{h,k})
  \big]
  \big|\\
& \quad
  \leq
  \sum_{l=1}^K
  \big\|
    \gamma_l(Y_l^{h,k})
  \big\|_{\mathbf{M}^{1,p}(H)}
  \Big\|
    \int_0^T
      \big(
        X_t-\tilde{X}_t^{h,k}
      \big)
    \diff{\nu_{l,t}}
  \Big\|_{\mathbf{M}^{1,p}(H)^*}\\
& \quad
  \leq
  \sum_{l=1}^K
  \Big(
    \sup_{h,k\in(0,1)}
    \big\|
      \gamma_l(Y_l^{h,k})
    \big\|_{\mathbf{M}^{1,p}(H)}
  \Big)
  \big\|
    X-\tilde{X}^{h,k}
  \big\|_{L_{\nu_l}^1(0,T;\mathbf{M}^{1,p}(H)^*)}.
\end{align*}
Here 
$
  \gamma_l\in
  \mathcal{G}_{\mathrm{p}}^{1,r}
  (H^{K+1};H)
$
and
$
  Y_l^{h,k} \in \mathbf{M}^{1,rp}(H^{K+1})
$ with
$r=\sum_{i=1}^K m_i -1$. Therefore
\cite{AnderssonKruseLarsson}*{Lemma 3.3} applied with $U=H^{K+1}$ and $V=H$
gives for
$
  l
  \in
  \{
    1,\dots,K
  \}
$ the bound
\begin{align*}
  \sup_{h,k\in(0,1)}
  \big\|
    \gamma_l(Y_l^{h,k})
  \big\|_{\mathbf{M}^{1,p}(H)}
  \leq
  C_l
  \Big(
    1
    +
    \sup_{h,k\in(0,1)}
    \big\|
      Y_l^{h,k}
    \big\|_{\mathbf{M}^{1,rp}(H^{K+1})}^r
  \Big).
\end{align*}
Propositions \ref{lemma:Stability2} and \ref{lemma:Stability3} ensure that
\begin{align*}
&  \sup_{h,k\in(0,1)}
  \big\|
    \gamma_l(Y_l^{h,k})
  \big\|_{\mathbf{M}^{1,p}(H)}\\
& \leq
  \tilde{C}_l
  \Big(
    1
    +
    \sum_{i=1}^K
    \Big(
      \big\|
        X
      \big\|_{L_{\nu_i}^1(0,T;\mathbf{M}^{1,rp,p}(H))}^r
      +
      \sup_{h,k\in(0,1)}
      \big\|
        \tilde{X}^{h,k}
      \big\|_{L_{\nu_i}^1(0,T;\mathbf{M}^{1,rp,p}(H))}^r
    \Big)
  \Big)
  <\infty.
\end{align*}

Let $\tilde{X}$ be the process $\tilde{X}_t=X_{t_n}$ for $t\in[t_n,t_{n+1})$, $n\in\{0,\dots,N-1\}$. Proposition \ref{prop:Holder} and Lemma \ref{thm:strong2} give, for $l\in \{1,\dots,K\}$,
\begin{align*}
& \big\|
    X-\tilde{X}^{h,k}
  \big\|_{L_{\nu_l}^1(0,T;\mathbf{M}^{1,p}(H)^*)}\\
& \quad
  \leq
  \big\|
    X-\tilde{X}
  \big\|_{L_{\nu_l}^1(0,T;\mathbf{M}^{1,p}(H)^*)}
  +
  \big\|
    \tilde{X}-\tilde{X}^{h,k}
  \big\|_{L_{\nu_l}^1(0,T;\mathbf{M}^{1,p}(H)^*)}
  \lesssim
  h^{2\gamma}
  +
  k^{\rho\gamma}.
\end{align*}
This completes the proof.
\end{proof}

Finally, we formulate a corollary of Theorem~\ref{thm:main} that can
be used to prove convergence of covariances and higher order
statistics of approximate solutions.  We demonstrate this for
covariances;  higher order statistics can be treated in a similar
way.
\begin{corollary}
\label{cor:main}
Let $X$ and $X^{h,k}$ be the solutions to \eqref{eq:SVIE} and \eqref{eq:approx}, respectively. Let $\tilde{X}_t^{h,k}=X_n^{h,k}$, for $t\in[t_n,t_{n+1})$, $n\in\{0,\dots,N-1\}$ and $\tilde{X}_t^{h,k}=X_N^{h,k}$, for $t\in[t_N,T]$. For $K\geq1$, $\phi_1,\dots,\phi_K\in H$, $t_1,\dots,t_K\in(0,T]$, $\gamma\in[0,\beta)$, there exists $C>0$ such that
\begin{align*}
  \Big|
  \E
  \Big[
    \prod_{i=1}^K
    \big\langle
      X_{t_i},\phi_i
    \big\rangle
    -
    \prod_{i=1}^K
    \big\langle
      \tilde{X}_{t_i}^{h,k},\phi_i
    \big\rangle
  \Big]
  \Big|
  \leq
  C
  \big(
    h^{2\gamma}+k^{\rho\gamma}
  \big)
  ,\quad
  h,k\in(0,1).
\end{align*}
In particular, for $\phi_1,\phi_2\in H$, $t_1,t_2\in(0,T]$, it holds that
\begin{align*}
& \big|
   \mathrm{Cov}
  \big(
    \big\langle
      X_{t_1},\phi_1
    \big\rangle
    ,
    \big\langle
      X_{t_2},\phi_2
    \big\rangle
  \big)
  -
  \mathrm{Cov}
  \big(
    \big\langle
      \tilde{X}_{t_1}^{h,k},\phi_1
    \big\rangle
    ,
    \big\langle
      \tilde{X}_{t_2}^{h,k},\phi_2
    \big\rangle
  \big)
  \big|\\
&\qquad
  \leq
  C
  \big(
    h^{2\gamma}+k^{\rho\gamma}
  \big)
  ,\quad
  h,k\in(0,1).
\end{align*}
\end{corollary}

\begin{proof}
  The first statement follows from Theorem \ref{thm:main} by setting
  $\varphi_i=\langle\phi_i,\cdot\rangle$, $\nu_i=\delta_{t_i}$,
  $i\in\{1,\dots,K\}$, where $\delta_{t_i}$ is the Dirac measure
  concentrated at $t_i$. The second is a consequence of the first and
  the fact that
\begin{align*}
&\mathrm{Cov}
  \big(
    \big\langle
      X_{t_1},\phi_1
    \big\rangle
    ,
    \big\langle
      X_{t_2},\phi_2
    \big\rangle
  \big)
  -
  \mathrm{Cov}
  \big(
    \big\langle
      \tilde{X}_{t_1}^{h,k},\phi_1
    \big\rangle
    ,
    \big\langle
      \tilde{X}_{t_2}^{h,k},\phi_2
    \big\rangle
  \big)\\
&\quad
  =
    \E
    \big[
    \big\langle
      X_{t_1},\phi_1
    \big\rangle
    \big\langle
      X_{t_2},\phi_2
    \big\rangle
    \big]
    -
    \E
    \big[
    \big\langle
      \tilde{X}_{t_1}^{h,k},\phi_1
    \big\rangle
    \big\langle
      \tilde{X}_{t_2}^{h,k},\phi_2
    \big\rangle
    \big]\\
& \qquad
  -
    \E
    \big[
    \big\langle
      X_{t_1},\phi_1
    \big\rangle
    -
    \big\langle
      \tilde{X}_{t_1}^{h,k},\phi_1
    \big\rangle
    \big]
    \E
    \big[
    \big\langle
      X_{t_2},\phi_2
    \big\rangle
    \big]\\
&\qquad
 -
  \E
  \big[
    \big\langle
      \tilde{X}_{t_1}^{h,k},\phi_1
    \big\rangle
  \big]
  \E
  \big[
    \big\langle
      X_{t_2},\phi_2
    \big\rangle
    -
    \big\langle
      \tilde{X}_{t_2}^{h,k},\phi_2
    \big\rangle
  \big].
\end{align*}
\end{proof}

\section{Examples}
\label{sec5}
In this section we consider two different types of equations and write them in the abstract form of Section \ref{sec1}. We verify the abstract assumptions in both cases. Numerical approximation by the finite element method and suitable time discretization schemes are proved to satisfy the assumptions of Section \ref{sec:weak}. We start with parabolic stochastic partial differential equations and continue with Volterra equations in a separate subsection.

\subsection{Stochastic parabolic partial differential equations}
\label{sec5:1}
Let $\D\subset\R^d$ for $d=1,2,3$ be a convex polygonal domain. Let $\Delta=\sum_{i=1}^d\frac{\partial^2}{\partial x_i^2}$ be the Laplace operator and $f\in\Cb^2(\R;\R)$. We consider the stochastic partial differential equation:
\begin{equation*}
\begin{array}{ll}
  \dot{u}(t,x)
  =
  \Delta u(t,x)
  +
  f(u(t,x))
  +
  \dot{\eta}(t,x),
  \quad
& (t,x)\in(0,T]\times \D,\\
  u(t,x)=0,
  \quad
& (t,x)\in(0,T]\times\partial \D,\\
  u(0,x)=u_0(x),
  \quad
& x\in \D.
\end{array}
\end{equation*}
The noise $\dot{\eta}$ is not well defined as a function, as it is written, but makes sense as a random measure. We will study this equation in the abstract framework of Section ~\ref{sec1}. Let $H=L^2(\D)$, $A\colon\D(A)\subset H\rightarrow H$ be given by $A=-\Delta$ with $\D(A)=H_0^1(\D)\cap H^2(\D)$. Let $(S_t)_{t\in[0,T]}$ denote the analytic semigroup $S_t=e^{-tA}$ of bounded linear operators generated by $-A$. Assumption \ref{as:S} is satisfied with $\rho=1$, as is easily seen by a spectral argument. The drift $F\colon H\rightarrow H$ is the Nemytskii operator determined by the action $(F(g))(x)=f(g(x))$, $x\in\D$, $g\in H$. Assumption \eqref{as:F} for $F$ is verified in \cite{Wang2014} for $\delta=\frac{d}2+\epsilon$.

Let $(\mathcal{T}_h)_{h\in(0,1)}$ denote a family of regular triangulations of $\D$ where $h$ denotes the maximal mesh size. Let $(V_h)_{h\in[0,1]}$ be the finite element spaces of continuous piecewise linear functions with respect to $(\mathcal{T}_h)_{h\in(0,1)}$ and $P_h\colon H\to V_h$ be the orthogonal projector. The operators $A_h\colon V_h\rightarrow V_h$ are uniquely determined by
\begin{align*}
  \big\langle
    A_h \phi_h,\psi_h
  \big\rangle
  =
  \big\langle
    \nabla \phi_h,\nabla \psi_h
  \big\rangle,
  \quad
  \forall \phi_h,\psi_h\in V_h\subset \dot{H}^1.
\end{align*}
\begin{remark}
If the domain $\D$ is such that the pairs of eigenvalues and eigenfunctions $(\lambda_n,e_n)_{n\in\N}$ of $A$ are known, e.g., $\D=[0,1]^d$, then instead of finite element discretization one can consider a spectral Galerkin approximation. Let the eigenvalues be ordered in increasing order so that $\lambda_n\leq \lambda_{n+1}$ for every $n\in\N$. Further, let $h=\lambda_{N+1}^{-\frac12}$ and $V_h=\mathrm{span}\{\phi_n:n\leq N\}$. By $P_h\colon H\to V_h$ we denote the orthogonal projector and we define $A_h=AP_h=P_hA=P_hAP_h$.
\end{remark}

We discretize in time by a semi-implicit Euler-Maruyama method. By defining $B_1^{h,k}=(I+kA_h)^{-1}P_h$ and $B_n^{h,k}=(B_1^{h,k})^n$ for $n\geq1$, the discrete solutions $(X_n^{h,k})_{n=0}^N$ are recursively given by
\begin{align*}
& X_{n+1}^{h,k}
  =
  B_1^{h,k}X_n^{h,k}+kB_1^{h,k}F(X_n^{h,k})
  +
  \int_{t_n}^{t_{n+1}}
    B_1^{h,k}
  \diff{W_s},
  \quad
  n=0,\dots,N-1,\\
& X_0^{h,k}=P_hx_0.
\end{align*}
Iterating the scheme gives the discrete variation of constants formula \eqref{eq:approx}. For both finite element and spectral approximation the assumptions \eqref{as:B}, \eqref{as:E1}, \eqref{as:E2}, \eqref{as:e}, are valid, see, e.g., \cite{thomee2006}. For a proof of \eqref{ineq:Etilde}, see \cite{AnderssonKruseLarsson}*{Lemma 5.1}.

\subsection{Stochastic Volterra integro-differential equations}
\label{sec5:2}
Consider the semi-linear stochastic Volterra type equation
\begin{equation}
\begin{aligned}
  \dot{u}(t,x)
& =
  \int_0^t
    b(t-s)\Delta u(t,x)
  \diff{s}
  +
  f(u(t,x))
  +
  \dot{\eta}(t,x),
&&\quad
  (t,x)\in(0,T]\times\D,\\
  u(t,x)
& =0,
&&\quad
  (t,x)\in(0,T]\times\partial\D,\\
  u(0,x)
& =u_0,
&&\quad
  x\in\D.
\end{aligned}
\end{equation}
We assume that the kernel $b\in L_{\mathrm{loc}}^1(\R_+)$ is $4$-monotone; that is, $b$ is twice continuously differentiable on $(0,\infty)$, $(-1)^nb^{(n)}(t)\ge 0$ for $t>0$, $0\le n\le 2$, and $b^{(2)}$ is non-increasing and convex.
We suppose further that $\lim_{t\to \infty}b(t)=0$ and
\begin{align}\label{eq:b-smooth}
  \limsup_{t\rightarrow 0,\infty}
  \Big(
    \frac 1t
    \int_0^t
      s b(s)
    \diff{s}
  \Big)
  \Big/
  \Big(
    \int_0^t
      -s \dot{b}(s)
    \diff{s}
  \Big)
     < +\infty.
\end{align}
In this case it follows from \cite{Prussbook}*{Proposition 3.10} that the parameter $\rho$ in Assumption \ref{as:B} is given by
\begin{equation}\label{eq:sector}
  \rho
  =
  1
  +
  \frac{2}{\pi}
  \sup
  \{ | \mathrm{arg} \, \hat b(\lambda) |: \; \text{Re}\,\lambda >0 \}
  \in (1,2),
\end{equation}
where $\hat{b}$ denotes the Laplace transform of $b$. Finally, in order to be able to use non-smooth data estimates for the deterministic problem we suppose that $\hat{b}$ can be extended to an analytic function in a sector $\Sigma_{\theta}=\{z\in \mathbb{C}: |\mathrm{arg} \,z|<\theta\}$ with $\theta>\frac{\pi}{2}$ and
$|\hat{b}^{(k)}(z)|\le C|z|^{1-\rho-k}$, $k=0,1$, $z\in \Sigma_{\theta}$.
An important example is the kernel $b(t)=\frac{1}{\Gamma(\rho-1)}t^{\rho-2}e^{-\eta t}$, for some $\rho\in(1,2)$ and $\eta\ge 0$. When $\eta=0$, then the corresponding equation can be viewed as a fractional-in-time stochastic equation.

We write the equation in the abstract It\={o} form \eqref{eq:SVDE} with $A$, $F$, $W$, $x_0$ as in Subsection \ref{sec5:1}. Here one needs $\delta=\frac{d}2+\epsilon<\frac2\rho$ and this requires $\rho<\tfrac43$ and $\epsilon$ small in the case $d=3$ but causes no restrictions in the case $d=1,2$. Under the above assumptions there exist a resolvent family of operators $(S_t)_{t\in[0,T]}$ defined by the strong operator limit
\begin{align}
\label{eq:explicit_S}
  S_t
  =
  \sum_{j=1}^\infty
    s_{j,t}\,
    (e_j\otimes e_j);
    \quad
  \dot{s}_{j,t}
  +
  \lambda_j
  \int_0^t
    b(t-r)
    s_{j,r}
  \diff{r}
  =0
  ,\
  t>0
  ;\quad
  s_{j,0}=1.
\end{align}
Here $(\lambda_j,e_j)_{j\in\N}$ are the eigenpairs of $A$. The operator family $(S_t)_{t\in[0,T]}$ does not possess the semigroup property because of the presence of the memory term. It is the solution operator to the abstract linear homogeneous problem
\begin{align*}
  \dot{Y}_t
  +
  \int_0^t
    b(t-s)AY_s
  \diff{s}
  =
  0,\
  t\in[0,T];\quad Y_0=y_0,
\end{align*}
i.e., $Y_t=S_ty_0$. The inhomogeneous problem with right hand side $g(t)$ for Bochner integrable $g\colon[0,T]\rightarrow H$ is solved by the variation of constants formula
\begin{align*}
  Y_t=S_ty_0+\int_0^tS_{t-s}g(s)\diff{s},\quad t\in[0,T].
\end{align*}
By \cite{BaeumerGeissertKovacs2014}*{Lemma 4.4} condition \eqref{as:S} holds for $S$. Thus the setting of Section \ref{sec1} is applicable.

We now turn our attention to the numerical approximation by presenting the convolution quadrature that we use, which was introduced by Lubich \cite{LubichI,LubichII}. Let $(\omega_j^k)_{j\in\N}$ be weights determined by
\begin{align*}
  \hat{b}
  \Big(
    \frac{1-z}k
  \Big)
  =
  \sum_{j=0}^\infty
    \omega_j^k z^j,
  \quad
  |z|<1.
\end{align*}
Then we use the approximation
\begin{align*}
  \sum_{j=1}^n
    \omega_{n-j}^k
    f(t_j)
  \sim
  \int_0^{t_n}
    b(t_n-s)f(s)
  \diff{s},
  \quad
  f\in\C(0,T;\R).
\end{align*}
To discretize the time derivative we use a backward Euler method, which is explicit in the semilinear term $F$. Our fully discrete scheme then reads:
\begin{align*}
& X_{n+1}^{h,k}-X_n^{h,k}
  +
  k
  \sum_{j=1}^{n+1}
    \omega_{n+1-j}^k
    A_h X_j^{h,k}\\
& \qquad\qquad\qquad
  =
  kP_hF(X_n^{h,k})
  +
  \int_{t_{n}}^{t_{n+1}}
    P_h
  \diff{W_t},
  \quad
  n=0,\dots,N-1,\\
& X_0^{h,k}=P_hx_0.
\end{align*}
It is possible to write $(X_n^{h,k})_{n=0}^N$ as a variation of
constants formula \eqref{eq:approx}. Indeed, it is shown in
\cite{kovacs2013} that one has the explicit representation
\begin{align*}
  B_n^{h,k}
  =
  \int_0^\infty
    S_{ks}^hP_h\frac{e^{-s}s^{n-1}}{(n-1)!}
  \diff{s},
  \quad
  n\geq1,
\end{align*}
where
\begin{align*}
  S_t^h
  =
  \sum_{j=1}^{N_h}
    s_{j,t}^h\,(e_j^h\otimes e_j^h)P_h;
  \quad
  \dot{s}_{j,t}^h
  +
  \lambda_j^h
  \int_0^t
    b(t-r)
    s_{j,r}^h
  \diff{r}
  =0,
  \
  t>0
  ;\quad
  s_{j,0}^h=1,
\end{align*}
and $(\lambda_j^h,e_j^h)_{j=1}^{N_h}$ are the eigenpairs corresponding
to $A_h$. The stability \eqref{as:B} holds by
\cite{kovacs2014}*{Theorem 3.1} and the smooth data error estimate
\eqref{as:E1} was proved in \cite{kovacs2013}*{Remark 5.3}. It remains
to verify \eqref{as:E2}. By \cite{kovacs2014}*{Theorem 3.1} there
exist $\tilde{C}$ so that
\begin{align*}
  \big\|
    E_n^{h,k}
  \big\|_{\LB}
  \leq
  \tilde{C}
  t_n^{-\frac{\delta}2}
  \big(
    h^{\frac\delta\rho}+k^{\frac{\delta}2}
  \big),
  \quad
  0\leq \delta\leq2,\
  n=1,\dots,N.
\end{align*}
Let $0\leq\delta\leq2$. Interpolation with $0\leq{s}\leq1$ yields
\begin{align*}
  \big\|
    A^{\frac{s}{2\rho}}E_{n,\theta}
  \big\|_{\LB}
& \leq
  \big\|
    E_n^{h,k}
  \big\|_{\LB}^{1-{s}}
  \big\|
    A^{\frac1{2\rho}}E_n^{h,k}
  \big\|_{\LB}^{{s}}\\
& \leq
  \big\|
    E_n^{h,k}
  \big\|_{\LB}^{1-{s}}
  \Big(
    \big\|
      A^\frac1{2\rho} S_{t_n}
    \big\|_{\LB}
    +
    \big\|
      A^\frac1{2\rho} B_n^{h,k}
    \big\|_{\LB}
  \Big)^{{s}}\\
& \leq
  \Big(
    \tilde{C}
    t_n^{-\frac{\delta}2}
    \big(
      h^{\frac\delta\rho}
      +
      k^{\frac{\delta}2}
    \big)
  \Big)^{1-{s}}
  \big(2L_{\frac12}
    t_n^{-\frac12}
  \big)^{{s}}\\
& \leq
  \tilde{C}^{1-{s}}(2L_\frac12)^{s}
  t_n^{-\frac{\delta(1-{s})+{s}}2}
  \big(
    h^{\frac{\delta(1-{s})}\rho}
    +
    k^{\frac{\delta(1-{s})}2}
  \big).
\end{align*}
Setting $\sigma=\delta(1-{s})$ and $R_{s} =\tilde{C}^{1-{s}}(2L_{\frac12})^{{s}}$ yields the estimate
\begin{align*}
  \big\|
    A^\frac{s}{2\rho}
    E_n^{h,k}
  \big\|_{\LB}
  \leq
  R_{s}
  t_n^{-\frac{\sigma+{s}}2}
  \big(
    h^{\frac\sigma\rho}+k^{\frac{\sigma}2}
  \big),
  \quad
  0\leq \sigma\leq2,\
  0\leq {s}\leq1-\frac\sigma2,
\end{align*}
for $n=1,\dots,N$. Therefore \eqref{as:E2} holds.

\subsection*{Acknowledgement}
We thank Arnaud Debussche for suggesting the consideration of
covariances and higher order statistics.  This motivated us to extend
Theorem 4.7 from its original statement with $K=1$ to $K\ge 1$ and
also to include Corollary \ref{cor:main}. The second author was
partially supported by the Marsden Fund project number UOO1418.

\def\polhk#1{\setbox0=\hbox{#1}{\ooalign{\hidewidth
  \lower1.5ex\hbox{`}\hidewidth\crcr\unhbox0}}}

%\bibliography{lit}
%\bibliographystyle{plain}

\end{document}